\documentclass{amsart}
\usepackage{rotating} 
\usepackage{enumerate}
\usepackage{amsfonts}
\usepackage{amsmath}
\usepackage{amssymb}
\usepackage{textcomp}
\usepackage{float,lscape}
\newtheorem{theorem}{Theorem}[section]

\newtheorem{lemma}[theorem]{Lemma}

\newtheorem{proposition}[theorem]{Proposition}

\newtheorem{corollary}[theorem]{Corollary}

\theoremstyle{definition}

\newtheorem{definition}[theorem]{Definition}

\newtheorem{example}[theorem]{Example}

\newtheorem{remark}[theorem]{Remark}

\newtheorem *{Theorem 1}{Theorem 1}
\newtheorem *{Theorem 2}{Theorem 2}
\newtheorem *{Theorem 5}{Theorem 5}
\newtheorem *{Theorem 3}{Theorem 3}
\newtheorem *{Theorem 4}{Theorem 4}
\newtheorem *{Problem 1}{Problem 1}

\newcommand{\C}{{\mathbb C}}

\newcommand{\RNum}[1]{\uppercase\expandafter{\romannumeral #1\relax}}
\newcommand{\rNum}[1]{\lowercase\expandafter{\romannumeral #1\relax}}

\begin{document}
\title[Simple twisted group algebras of dimension $p^4$]{Simple twisted group algebras of dimension $p^4$ and their semi-centers}
\address{Institute of Algebra and Number Theory, Pfaffenwaldring 57\\
University of Stuttgart, Stuttgart 70569, Germany}
\author{Ofir Schnabel}
\email{os2519@yahoo.com}
\thanks{This work has been supported by the Minerva Stiftung}

\begin{abstract}
For simple twisted group algebra over a group $G$, if $G^{\shortmid}$ is  Hall subgroup of $G$ then the semi-center is simple.
Simple twisted groups algebras correspond to groups of central type.
We classify all groups of central type of order $p^4$ where $p$ is prime and use this  to show that
for odd primes $p$ there exists a unique group $G$ of order $p^4$ such that 
there exists simple twisted group algebra over $G$ with a commutative semi-center. 
Moreover, if $1< |G|< 64$, then the semi-center of simple twisted group algebras over $G$ is non-commutative and this bounds are strict.
\end{abstract}

\maketitle
\bibliographystyle{abbrv}
\section{Introduction}\pagenumbering{arabic} \setcounter{page}{1}
Let $G$ be a finite group. A \textit{twisted group algebra}
over $\mathbb{C}$, denoted by $\mathbb{C}^fG$, is an associative algebra with
basis $\{u_g\}_{g\in G}$. The multiplication is defined on basis
elements as follows. For any $x,y\in G$
\begin{equation}\label{eq:multtga}
u_xu_y=f(x,y)u_{xy}, 
\end{equation}
and it is extended distributively. Here, $f\in Z^2(G,\mathbb{C}^*)$. That is, $f$ is a $2$-cocycle.
By a generalization of
Maschke's theorem, complex twisted group algebras are semi-simple
\cite[Theorem 3.2.10]{karpilovsky} and therefore, by the
Artin-Wedderburn theorem it is isomorphic to a direct sum of matrix
algebras. 
For a non-trivial group $G$, the group algebra $\mathbb{C}G$ is not simple. However, it turns
out that $2$-cocycles, $f\in Z^2(G,\mathbb{C}^*)$ may exist such that
$\mathbb{C}^fG$ is simple. A group $G$ admitting such a phenomenon is called of
{\it central type}
\footnote{ Classically, a group $G$ is said to be of central
type if it has a faithful irreducible character
$\chi$ such that $\chi(1)^2 = |G:Z|$, where $Z = Z(G)$.
In this paper however we call $G/Z$ (which is sometimes called ``Central-type factor group'') of central type.}, and the 2-cocycle $f\in Z^2(G,\C^*)$
is called {\it nondegenerate}. Evidently, the
size of any group of central type is a square.

It turns out that twisted group algebras and groups of central type play a major role in the understanding of other algebraic objects,
as graded algebras (see \cite{MR1941224},\cite{EK13},\cite{ginosargradings}), semi-invariants of matrices (see \cite{ginosar2012semi})
and twisted category algebras (see \cite{Danz},\cite{Linckelmann}).

The classification of the {\it abelian} groups of central type is well known (see \S\ref{main}).
It is much harder to understand non-abelian groups of central type.
In fact, their complete classification up to isomorphism looks like an impossible task.
In \cite{MR652860} R. Howlett and M. Isaacs proved, using the classification of finite simple groups,
that groups of central type are solvable.
Another important result is that a cocycle $f\in Z^2(G,\C^*)$ is nondegenerate if and only if
its restriction to any Hall subgroup is nondegenerate (see \cite[Corollary 4]{DeMeyer}, \cite[Lemma 2.7]{MR2731925}).
In \cite{ginosargradings} we classify all the groups of central type of order $n^2$ where $n$ is a square-free number.
Also, in \cite[Theorem C]{ginosar2012semi} we show that for any group of central type $G$, $Z(G)$ embeds into $\hat{G}$,
where $Z(G)$ is the center of $G$ and $\hat{G}$ is the group of $1$-dimensional characters of $G$.

In the first part of this paper we classify the groups of central type of order $p^4$ where $p$ is a prime number.
In order to do so we use the classification of groups of order $p^4$ \cite[p.100-102]{Burnside}. For convenience we 
change the notations from \cite[p.100-102]{Burnside} and rewrite these groups. The corresponding tables are Table~\ref{tab:p4} for the odd case and
Table~\ref{tab:16} for the even case.
The following two theorems classify the groups of central type of order $p^4$ where $p$ is a prime number.
\begin{Theorem 1}
Let $p$ be an odd prime.
Up to isomorphism, the groups of central type of order $p^4$ are exactly the groups  $G_{(\text{iii})},G_{(\text{v})},G_{(\text{viii})},G_{(\text{xiv})},G_{(\text{xv})}$ in Table~\ref{tab:p4}.
\end{Theorem 1}
\begin{Theorem 2}
Up to isomorphism, the groups of central type of order $16$ are exactly the groups  $G_{(\text{iii})},G_{(\text{v})},G_{(\text{ix})},G_{(\text{x})}$ in Table~\ref{tab:16}.
\end{Theorem 2}
In the proof of Theorem 1 and Theorem 2 we are basically using the following three different methods:\\
Let $G$ be a group of order $p^4$ which is \textit{not} of central type. Using Lemma~\ref{lemma:trivcentralizer}, Lemma~\ref{lemma':toobycyc} and \cite[Theorem C]{ginosar2012semi}), which
deal with groups of central type of any order, we try to show that $G$ is not a group of central type.
If this fails, we show, using alternative definition of nondegeneracy (see Definition~\ref{def:altdef}),
that for any $f\in Z^2(G,\C ^*)$, the dimension of the center of $\C ^fG$ is greater than $1$.\\
In order to show that a group $G$ of order $p^4$ is of central type we construct a nondegenerate cohomology class $[f]\in H^2(G,\C ^*)$.
We do this by using a construction of crossed product which arises from an action of a group $H$ on a ring $R$ by an automorphism of $R$.

The twisted group
algebra $\mathbb{C}^fG$ is equipped with a $\mathbb{C}G$-module
structure defined by,
\begin{equation}\label{eq:conma}
g(u_h):=u_gu_hu_g^{-1}.
\end{equation}
This furnishes $\mathbb{C}^fG$ with a
$G$-module algebra structure.

For any $G$-module-algebra $A$ over a field $K$,
a nonzero element $a\in A$ is called {\it semi-invariant} if there exists $\lambda \in
\hat{G}:=\text{Hom}(G,K^*)$ (the {\it{weight}} of $a$), such
that for any $g\in G$, $$g(a)=\lambda (g)\cdot a.$$ The subspace
spanned by all the semi-invariant elements is a subalgebra of $A$ called the \textit{semi-center} of $A$,
and denoted by Sz$(A):=$Sz$_G(A)$.
Then $$\text{Sz}(A)=\oplus_{\lambda\in \hat{G}}A _{\lambda},$$ where $A _{\lambda}$ is the
subspace of all the semi-invariant elements in $A$ of weight $\lambda$ (and zero).
This is a natural grading of the semi-center
by the group $\hat{G}$ of 1-dimensional $G$-characters.

In the second part of this paper we study the semi-center of twisted group algebras under the conjugation action~\eqref{eq:conma}.
Semi-invariants and the semi-center of module-algebras are investigated mainly with regard to enveloping
algebras of finite dimensional Lie algebras, see, e.g.
\cite{vernick, dixmier, smith},
and with regard to group algebras, see, e.g. \cite{MR1476385,Wauters1999,Wauters2001}.
In \cite{ginosar2012semi} Y. Ginosar and the author generalized a result of D. Passman and P. Wauters (see \cite{MR1476385}), by showing that if
the Artin-Wedderburn decomposition of $\mathbb{C}^fG$ is
$$\mathbb{C}^fG=\oplus _{i=1}^r M_{n_i}(\mathbb{C}),$$
then
\begin{equation}\label{eq:sztga}
\text{Sz}\left(\mathbb{C}^fG\right)=\oplus _{i=1}^r \mathbb{C}^{f_i}G_i,
\end{equation}
where the $G_i$'s are subgroups of $\hat{G}$, and $f_i\in Z^2(G_i,\mathbb{C}^*)$.
In particular, if $\mathbb{C}^fG$ is simple then its semi-center is a twisted group algebra, that is,
\begin{equation}\label{eq:scsimp}
\text{Sz}(\mathbb{C}^fG)=\mathbb{C}^{\hat{f}}\hat{G}.
\end{equation}

We study the following problem.
\begin{Problem 1}
Which cocycles $f\in Z^2(G,\C^*)$ admit
\begin{enumerate}[(i)]
\item A commutative semi-center of $\C ^fG$?
\item A simple semi-center of $\C ^fG$?
\end{enumerate}
\end{Problem 1}
By~\eqref{eq:sztga} a necessary condition for affirmative answer to Problem 1(ii) is that $f$ is nondegenerate.
Regarding Problem 1(ii) we prove the following
\begin{Theorem 3}
Let $\C ^fG$ be a simple twisted group algebra. If the restriction of the nondegenerate cocycle $f$ to $G^{\shortmid}$
is also nondegenerate then the semi-center of $\C ^fG$ is simple.
In particular, if $G^{\shortmid}$ is a Hall subgroup of
$G$, then the semi-center of $\C ^fG$ is simple.
\end{Theorem 3}
Problem 1(i) is particularly interesting when restricting to nondegenerate cocycles.
By Theorem 1, the group $G_{(\text{xv})}$ in Table~\ref{tab:p4} is of central type.
In Theorem~\ref{th:existp4} we show that there exist a nondegenerate cocycle $f\in Z^2(G_{(\text{xv})},\mathbb{C}^*)$
such that Sz$(\mathbb{C}^fG_{(\text{xv})})$ is commutative and in Proposition~\ref{th:p^4} we show that if $G\not \cong G_{(\text{xv})}$ is a group of central type of order $p^4$,
then for any nondegenerate cocycle
$f \in Z^2(G,\C^*)$, Sz$(\mathbb{C}^fG)$ is non-commutative. We get the following theorem.
\begin{Theorem 4}
Let $G$ be a group of central type of order $p^4$ where $p$ is prime. There exists a nondegenerate cocycle $f\in Z^2(G,\mathbb{C}^*)$
such that Sz$(\mathbb{C}^fG)$ is commutative if and only if  $p$ is odd and $G$ is isomorphic to the group $G_{(\text{xv})}$ in Table~\ref{tab:p4}
\end{Theorem 4}
In view of Theorem 4 and Remark~\ref{remark:sf}, for any group of central type $G$, if\\
$1< |G|< 64$, then the semi-center of simple twisted group algebras over $G$ is non-commutative.
In the following theorem we present a group $G$ of order $64$ such that there exists a commutative semi-center of simple twisted group algebra over $G$.
\begin{Theorem 5}
Let
\begin{equation}
G=\langle x_1,x_2,x_3,x_4,x_5,x_6 \rangle,
\end{equation}
such that $ x_1,x_2,x_3$ are central and 
$$x_i^2=1,\quad [x_4,x_5]=x_1,\quad [x_4,x_6]=x_2,\quad [x_5,x_6]=x_3.$$
There exists a nondegenerate cocycle $f \in Z^2(G,\C^*)$
such that Sz$(\mathbb{C}^fG)$ is commutative.
\end{Theorem 5}
By Theorem 4 and Remark~\ref{remark:sf}, the group $G$ described in Theorem 5 is a minimal non-trivial group with the property that the semi-center of a simple
twisted group algebra over a group is commutative.

{\bf Acknowledgements.}
A part of this paper appears in the author's Ph.D. dissertation under the supervision of Y. Ginosar and the rest of this paper was influated by it.
The author is also grateful to
F. Cedo, M. Kochetov and S. Koenig for their valuable comments.
Lastly, I want to thank the referee for reviewing the first draft of this paper and for his valuable suggestions.
\section{preliminaries}\label{preliminaris}
For any $f\in Z^2(G,\mathbb{C}^*)$ define an antisymmetric form from
the set of commuting pairs in $G$ to $\mathbb{C}^*$ as follows:
\begin{equation}\label{eq:form}
\alpha _f(g_1,g_2):=f(g_1,g_2)f(g_2,g_1)^{-1}.
\end{equation}
If $G$ is abelian group then $\alpha _f$ determines the cohomology
class of $f$. By the definition of $\alpha _f$, if $[f_1]=[f_2]\in H^2(G,\mathbb{C})$, then for every commuting
elements $g_1,g_2\in G$, $\alpha _{f_1}(g_1,g_2)=\alpha
_{f_2}(g_1,g_2)$.
In other words, there is a well-defined function $[f]\rightarrow
\alpha _f$ from $H^2(G,\mathbb{C})$ to the antisymmetric forms.
Throughout this paper we will deal with nondegeneracy of cocycles $f$ using the notion of $f$-regularity.
\begin{definition}\cite[\S 2]{Oystaeyen}
Let $f\in Z^2(G,\mathbb{C}^*)$. An element $g\in G$ is called
\textbf{$f$-regular} if $\alpha _f(g,h)=1$ for every $h\in C_G(g)$ (the centralizer of $g\in
G$).
\end{definition}
In other words, an element $g\in G$ is $f$-regular if and only if
for every element $h\in C_{G}(g)$, $u_gu_h=u_hu_g$ (see~\eqref{eq:multtga}) in the twisted
group algebra $\mathbb{C}^fG$. 
It is easy to show that regularity is a class property, both cohomological and conjugacy.

It is well known that the dimension of the center of $\mathbb{C}^fG$ is equal to the number of $f$-regular conjugacy classes of $G$ (see \cite[Theorem 2.4]{Oystaeyen}).
This leads to an alternative definition of non-degeneracy of a 2-cocycle $f\in Z^2(G,\C ^*)$.
\begin{definition}\label{def:altdef}
A $2$-cocycle $f\in Z^2(G,\C ^*)$ is called {\it non-degenerate} if
$1\in G$ is the only $f$-regular element in $G$.
\end{definition}
We wish to formulate two known results which we will use in order to show that some groups are not of central type.
\begin{lemma}\label{lemma:trivcentralizer}
Let $G$ be a group of central type. Then, $C_G(g)$ is non-cyclic for any non-trivial element $g\in G$. 
\end{lemma}
\begin{proof}
Since for any cocycle $f\in Z^2(G,\mathbb{C}^*)$ and any $h\in G$, $u_h$ commutes with $u_{h^i}$,
we conclude that if $C_G(g)=\langle h \rangle$, then $g$ is $f$-regular.
In particular, in this case, there are no nondegenerate $f\in Z^2(G,\mathbb{C}^*)$. 
\end{proof}
By \cite[Proposition 1.2]{MR3210715} if there exists a subgroup $H$ of $G$ such that the 
restriction of $f$ to $H$ is trivial then $|H|\leq \sqrt{|G|}$. Since the cohomology of cyclic groups is trivial,
we obtain following lemma.
\begin{lemma}\label{lemma':toobycyc}
Let $G$ be a group of central type of order $n^2$. Then, the order of any element in $G$ is less or equal $n$. 
\end{lemma}
Let $G$ be a group. Denote by $\circ (g)$ the order of $g\in G$.
Since $\C$ is algebraically closed we may always assume that for any  $f\in Z^2(G,\C ^*)$,
\begin{equation}\label{eq:orderof}
 u_g^{\circ(g)}=1
\end{equation}
in the corresponding twisted group algebra $\C ^f G$.
Therefore, since for any  $f\in Z^2(G,\C ^*)$
$$[u_g,1]=1,$$
in the corresponding twisted group algebra $\C ^f G$,
the following lemma is clear.
\begin{lemma}\label{lemma:commupto}
Let $G$ be a finite group, let $f\in Z^2(G,\C ^*)$ and let $a,b\in G$ be commuting elements
such that in the corresponding twisted group algebra $\C ^f G$,
$$[u_a,u_b]=\lambda .$$
Then $\lambda $ is a root of unity of order dividing the greatest common divisor of $\circ(a)$ and $\circ(b)$.
\end{lemma}
The concept of $f$-regularity is generalized bellow.
\begin{definition}\label{def:nir}[N. Ben David]
Let $f\in Z^2(G,\mathbb{C}^*)$ and let $\lambda\in \hat{G}$. An
element $x\in G$ is called \textbf{$(\lambda,f)$-regular} if for any
$g\in C_G(x)$, $$\lambda (g)=\alpha _f(g,x).$$
\end{definition}
As before, $(\lambda, f)$-regularity is a class property (conjugacy
and cohomology). In particular, an element is $f$-regular if it is
$(1,f)$-regular.
Let $f\in Z^2(G,\mathbb{C}^*)$, $\lambda\in \hat{G}$ and let $x\in
G$ be a $(\lambda, f)$-regular element. Assume $T=\{1,t_2,\ldots
,t_n\}$ is a left transversal of $C_G(x)$. Denote
\begin{equation}\label{eq:span}
S_{(\lambda ,x)}=\sum_{i=1}^n \lambda ^{-1}(t_i)
u_{t_i}u_xu_{t_i}^{-1}.
\end{equation}
Proposition~\ref{th:weightspaces} which gives a complete
description of the weight spaces $(\mathbb{C}^fG) _{\lambda}$, and Lemma~\ref{remark:sc} claiming that
any central element in $G$ induce a semi-invariant element in $\mathbb{C}^fG$ were both proven in \cite{ginosar2012semi}.
\begin{proposition}\label{th:weightspaces}\cite[Proposition 6.2]{ginosar2012semi}
With the above notation,
\begin{equation}
(\mathbb{C}^fG) _{\lambda}=\text{span} _{\mathbb{C}}\{S_{(\lambda
,x)}|\hspace{1mm} x \hspace{1mm} \text{is}\hspace{1mm}
(\lambda,f)\text{-regular} \}.
\end{equation}
\end{proposition}
\begin{lemma}\label{remark:sc}
Let $G$ be a group with center $Z(G)$. For every element $x\in Z(G)$
and $f\in Z^2(G,\mathbb{C}^*)$ the element $u_x$ is semi-invariant
in $\mathbb{C}^fG$. In particular, for any abelian group $G$, and any 2-cocycle $f\in
Z^2(G,\mathbb{C}^*)$, $$\text{Sz}(\mathbb{C}^fG)=\mathbb{C}^fG.$$
\end{lemma}
We will also use in the sequel the following estimation of the dimension of the semi-center of twisted group algebras.
\begin{corollary}\label{th:theorem1}(see \cite[Corollary 6.5]{ginosar2012semi})
For any $f\in Z^2(G,\mathbb{C}^*)$, let 
$$\Gamma(f)=\bigcup_{\lambda\in\hat{G}} \Gamma (\lambda,f)$$
be a set of representatives of the
$(\lambda,f)$-regular conjugacy classes for all $\lambda \in\hat{G}$. Then
$$\dim_{\C}{\rm Sz}(\C^fG)=\sum _{x\in \Gamma(f)}[G:G' C_G(x)].$$
\end{corollary}
 \section{Classification of groups of central type of order $p^4$}\label{main}
In this section we prove Theorem 1 and Theorem 2.
Throughout this section we will use the classification (and numbering) of groups of order $p^4$, which can be found in Table~\ref{tab:p4}
for the odd case and Table~\ref{tab:16} for the even case.\\
At first we will handle the odd and the even case simultaneously (groups $G_{(\text{i})}$-$G_{(\text{vii})}$). However, at some point (from group $G_{(\text{viii})}$ and forward)
we will separate between the cases.

We start from the well known abelian case.
Abelian groups of central type are exactly groups of the form $A\times A$,
that is the direct sum of two copies of the same abelian group (see \cite[Theorem 5]{BSZ}).
In particular, the abelian groups of central type of order $p^4$ are
\begin{equation}
 C_p\times C_p \times C_p \times C_p \text{ and } C_{p^2}\times C_{p^2}.
\end{equation}
This shows that from the groups $G_{(\text{i})}$-$G_{(\text{v})}$ in both Table~\ref{tab:p4}
and Table~\ref{tab:16}, the groups of central type are $G_{(\text{iii})}$ and $G_{(\text{v})}$.

We now start dealing with the non-abelian case.
The group $G_{(\text{vi})}$ (in both tables) admits a cyclic subgroup of order $p^3$ and hence by Lemma~\ref{lemma':toobycyc} it is not of central type.

In the group $G_{(\text{vii})}$, the center is $\langle a \rangle \cong C_{p^2}$ and the commutator subgroup of $G_{(\text{vii})}$ is $\langle a^p \rangle$.
Therefore, the group of $1$-dimensional characters of $G_{(\text{vii})}$ is isomorphic to $C_p^3$.
Hence, there is no embedding of the center of $G_{(\text{vii})}$ in the group of $1$-dimensional characters of $G_{(\text{vii})}$.
Consequently, by \cite[Theorem C]{ginosar2012semi} this group is not of central type. 

It turns out that the group  $G_{(\text{viii})}$ is of central type in the odd case and is not of central type in the even case.
From now on we will separate between the odd case and the even case.
\subsection{$p$ is odd}
In this section we complete the proof of Theorem 1.
The numbering of the groups in this section is with respect to Table~\ref{tab:p4}.
We start by showing that the group  $G_{(\text{viii})}$ is of central type.
In order to do so we will use a construction of crossed product which arises from an action of a group $G$ on a ring $R$ by an automorphism of $R$.
\begin{theorem}
Let $G=\langle a,b \rangle$ be a group of order $p^4$ such that 
$$a^{p^2}=b^{p^2}=1,\quad [a,b]=a^p.$$
Then $G$ is a group of central type. 
\end{theorem}
\begin{proof}
We wish to construct a nondegenerate cohomology class $[f]\in H^2(G,\C^*)$.
This twisted group algebra $\mathbb{C}^fG$ is generated as algebra by $u_a,u_b$ such that
\begin{equation}\label{eq:basicrelations2}
u_a^{p^2}=u_b^{p^2}=1.
\end{equation}
We use a construction of crossed product in the following way.
Consider the following $C_{p^2}=\langle b \rangle$-action on the commutative group algebra $R_1=\mathbb{C}\langle a \rangle$.
\begin{align*}
\psi _b: & R_1  \rightarrow  R_1\\
& u_a\mapsto \zeta _{p^2} u_a^{p+1}
\end{align*}
where $\zeta _{p^2}$ is a root of unit of order $p^2$.
By setting $\psi_b(u_a^i)=\psi _b(u_a)^i$ we ensure that $\psi _b$ is a homomorphism of $R_1$.

Next, we check that the order of the automorphism $\psi _b$ is $p^2$.
By induction it is easy to show that 
\begin{equation}\label{eq:ind8}
\psi_b ^k(u_a)=\zeta _{p^2}^{\left( k+p \sum _{i=1}^{k-1} i\right)} u_a^{kp+1}=\zeta _{p^2}^{\left( k+\frac{k(k-1)p}{2} \right)} u_a^{kp+1}.
\end{equation}
Then, since $u_a$ has order $p^2$, if 
$$\psi_b ^k(u_a)=\zeta _{p^2}^{\left( k+\frac{k(k-1)p}{2} \right)} u_a^{kp+1}=u_a$$
then $k\equiv 0($mod $p)$.

Assume now that $k=ip$.
If $i\not \equiv 0($mod $p)$ then 
$$ k+\frac{k(k-1)p}{2}=ip+\frac{ip(ip-1)p}{2} \not \equiv 0(\text{mod }p^2).$$
Therefore, since $\zeta _{p^2}$ has order $p^2$, if $\psi_b ^k(u_a)=u_a$ then $k\equiv 0($mod $p^2)$.
Consequently, the order of $\psi_b$ is at least $p^2$.
By putting $k=p^2$ in~\eqref{eq:ind8} we get
$$\psi_b ^{p^2}(u_a)=\zeta _{p^2}^{\left( p^2+\frac{p^2(p^2-1)p}{2} \right)} u_a^{p^3+1}=\zeta _{p^2}^{p^2\left( 1+\frac{(p^2-1)p}{2} \right)}u_a=u_a.$$
We conclude that $\psi_b$ is an automorphism of $R_1$ of order $p^2$.

Since $\psi _b$ is an automorphism of the ring $R_1$, the ring $R$ generated by $u_a,u_b$ is a crossed product of 
$C_{p^2}=\langle u_b \rangle$ over $R_1$.
In particular it is an associative algebra. Hence the ring $R$ is an associative complex algebra generated by $u_a,u_b$ with the following relations:
$$ u_a^{p^2}=u_b^{p^2}=1,\quad  u_au_bu_a^{-1}=\zeta _{p^2}u_bu_a^p.$$
The associativity of $R$ ensures that there exists a cocycle $f\in Z^2(G,\mathbb{C}^*)$ which satisfies the above relations.
We show that $f$ is nondegenerate by showing that there are no non-trivial $f$-regular elements in $G$.
Since $\psi_b $ has order $p^2$,
$$\psi_b ^p(u_a)=u_b^{-p}u_au_b^p\neq u_a\Rightarrow [u_a,u_b^p]\neq 1.$$
Since $b^p\in C_G(a)$ we conclude that $a$ and $b^p$ are not $f$-regular.
On the other hand, by using the equality $u_a u_b u_a^{-1}=\zeta _{p^2}u_bu_a^p$, it is easy to prove by induction that
$$u_a^ku_bu_a^{-k}=\zeta _{p^2}^k u_bu_a^{kp}.$$
In particular, for $k=p$ we get
$$u_a^pu_bu_a^{-p}=\zeta _{p^2}^p u_bu_a^{p^2}=\zeta _{p^2}^p u_b\neq u_b\Rightarrow [u_a^p,u_b]\neq 1.$$
Since $a^p\in C_G(b)$ we conclude that $b$ and $a^p$ are not $f$-regular.

Now, let $x=a^ib^j \in G$ be an $f$-regular element.
If $j\not \equiv 0($mod $p)$ then $[u_x,u_a^p]\neq 1$ while $a^p\in C_G(x)$. Hence we may assume that 
$j\equiv 0($mod $p)$.\\
If $j\equiv 0($mod $p)$ and $j \not \equiv 0($mod $p^2)$, then $[u_x,u_a]\neq 1$ while $a\in C_G(x)$. Hence we may assume that $j=0$ and $x=a^i$.\\
If $i\not \equiv 0($mod $p)$ then $[u_x,u_b^p]\neq 1$ while $b^p\in C_G(x)$. Hence we may assume that  $i\equiv 0($mod $p)$.\\
But then if $i\equiv 0($mod $p)$ and $i \not \equiv 0($mod $p^2)$, then $[u_x,u_b]\neq 1$ while $b\in C_G(x)$ and therefore $x$ is trivial.
Consequently, $f$ is nondegenerate and $G$ is a group of central type.
\end{proof}
It turns out that the group  $G_{(\text{ix})}$ is of central type in the even case and is not of central type in the odd case.
The main reason for that is related to the following number-theoretical lemma.
\begin{lemma}\label{lemma:distinguishp2}
 Let $p$ be a prime number.
 Then, $p^3$ is a divisor of $(p+1)^p-1$ if and only if $p=2$.
\end{lemma}
\begin{proof}
One direction is clear.
Namely, $2^3=3^2-1$.\\
Assume now that $p$ is an odd prime. We prove that $(p+1)^p-1\equiv p^2 (\text{ mod }p^3)$.
By the binomial identity 
$$(p+1)^p-1=\sum _{i=1}^p \binom{p}{i}p^i.$$
Consequently,
$$(p+1)^p-1\equiv \binom{p}{2}p^2+p^2 (\text{ mod }p^3).$$
Since $p$ is odd we conclude that $\binom{p}{2}p^2\equiv 0 (\text{ mod }p^3)$ and the result follows.
\end{proof}

The following lemma shows that the group $G_{(\text{ix})}$ is not of central type.
\begin{lemma}\label{lemma:groupnine}
 Let $G=\langle a,b,c \rangle$ be a group of order $p^4$ such that
 $$a^{p^2}=b^p=c^p=1,\quad [a,b]=[b,c]=1,[a,c]=a^p.$$
 Then, for any cocycle $f\in Z^2(G,\mathbb{C}^*)$, the element $a^p$ is $f$-regular.
 In particular, $G$ is not a group of central type.
\end{lemma}
\begin{proof}
Let $f\in Z^2(G,\mathbb{C}^*)$, and assume that 
$$[u_a,u_b]=\delta, [u_b,u_c]=\lambda, [u_a,u_c]=\gamma u_a^p.$$
By Lemma~\ref{lemma:commupto}, $\delta$ and $\lambda $ are both $p$-th roots of unity.
Therefore,
\begin{equation}\label{eq:apisfreg3}
[u_a^p,u_b]=\delta ^p=1.
\end{equation}
Consequently, if $[u_a^p,u_c]=1$, then $a^p$ is $f$-regular and the proof is completed.
Using the equality $u_au_cu_a^{-1}=\gamma u_a^pu_c$ we compute
\begin{align}
& [u_a^k,u_c]=u_a^{k-1}u_au_cu_a^{-1}u_a^{-(k-1)}u_c^{-1}= u_a^{k-1}(u_au_cu_a^{-1})u_a^{-(k-1)}u_c^{-1}= \\ \nonumber
& u_a^{k-1}(\gamma u_a^pu_c)u_a^{-(k-1)}u_c^{-1}= \gamma u_a^p u_a^{k-1} u_cu_a^{-(k-1)}u_c^{-1}=\gamma u_a^p [u_a^{k-1},u_c].
\end{align}
Therefore, by induction 
\begin{equation}\label{eq:commom3}
[u_a^k,u_c]=\gamma^k u_a^{pk}.
\end{equation}
In particular, by putting $k=p$ in~\eqref{eq:commom2} we get
\begin{equation}\label{eq:apisfreg4}
[u_a^p,u_c]=\gamma^pu_a^{p^2}=\gamma^p. 
\end{equation}
Now, in order to prove that $u_{a^p}$ is $f$-regular we need to show that 
$\gamma$ is a $p$-th root of unity.

Since $a^p$ and $c$ commutes in $G$, then by Lemma~\ref{lemma:commupto}, $\gamma$ is a root of unity of order dividing $p^2$.
By raising both sides of 
$$u_cu_au_c^{-1}=\gamma ^{-1} u_a^{1-p}$$ to the $p+1$ power we get that 
$$u_cu_a^{p+1}u_c^{-1}=\gamma^{-(p+1)} u_a,$$ and therefore 
$$u_c^{-1}u_au_c=\gamma^{p+1} u_a^{p+1}.$$
From here, by induction it is easy to prove that
\begin{equation}\label{eq:eqfor9}
u_c^{-k}u_au_c^k=\gamma ^{\sum _{i=1}^{k}(p+1)^i} u_a^{\left({(p+1)}^k\right)}.
\end{equation}
Since $u_c^p=u_a^{p^2}=1$, then by putting $k=p$ in~\eqref{eq:eqfor9} we get
\begin{equation}
u_a=u_c^{-p}u_au_c^p=\gamma ^{\sum _{i=1}^{p}(p+1)^i} u_a^{\left({(p+1)}^p\right)}=\gamma ^{\frac{(p+1)[(p+1)^p-1]}{p}}u_a.
\end{equation}
By Lemma~\ref{lemma:distinguishp2}, $p^2$ is not a divisor of ${\frac{(p+1)^p-1}{p}}$ and therefore, since $p+1$ is prime to $p$, the order of $\gamma$
is not $p^2$. We conclude that $\gamma$ is a $p$-th root of unity and therefore by~\eqref{eq:apisfreg3} and~\eqref{eq:apisfreg4}
$a^p$ is $f$-regular which completes the proof.
\end{proof}
Similarly to the case of group $G_{(\text{ix})}$,
it turns out that the group $G_{(\text{x})}$  is of central type in the even case and is not of central type in the odd case.
\begin{lemma}\label{lemma:groupten}
 Let $G=\langle a,b,c \rangle$ be a group of order $p^4$ such that
 $$a^{p^2}=b^p=c^p=1,\quad [a,b]=[b,c]=1,[a,c]=b.$$
 Then, for any cocycle $f\in Z^2(G,\mathbb{C}^*)$, the element $a^p$ is $f$-regular.
 In particular, $G$ is not a group of central type.
\end{lemma}
\begin{proof}
Let $f\in Z^2(G,\mathbb{C}^*)$, and assume that $[u_a,u_b]=\delta$ and $[u_b,u_c]=\lambda$.
By Lemma~\ref{lemma:commupto}, $\delta $ and $\lambda $ are both $p$-th roots of unity.
Therefore, 
\begin{equation}\label{eq:apisfreg}
[u_a^p,u_b]=\delta ^p=1.
\end{equation}

Assume now that $[u_a,u_c]=\gamma u_b$.
We show that $\gamma$ is also a $p$-th root of unity. We use the equalities 
$$u_cu_a^{-1}u_c^{-1}=\gamma u_a^{-1}u_b \quad \text{and}\quad  u_bu_c^{-1}=\lambda ^{-1}u_c^{-1}u_b$$
in order to show the following equality:
\begin{align*}
& [u_a,u_c^k]=u_au_c^{k-1}u_cu_a^{-1}u_c^{-1}u_c^{-(k-1)}= u_au_c^{k-1}(u_cu_a^{-1}u_c^{-1})u_c^{-(k-1)}= \\ \nonumber
& u_au_c^{k-1}(\gamma u_a^{-1}u_b)u_c^{-(k-1)}= \gamma u_au_c^{k-1} u_a^{-1}u_c^{-(k-1)}\lambda ^{-(k-1)}u_b=\gamma \lambda ^{-(k-1)}[u_a,u_c^{k-1}]u_b. 
\end{align*}
Therefore, by induction 
\begin{equation}\label{eq:commom}
[u_a,u_c^k]=\gamma^k \lambda ^{-(\sum_{i=1}^k i-1)}u_b^k=\gamma^k \lambda ^{-\frac{k(k-1)}{2}}u_b^k.
\end{equation}
Consequently, since $p$ is odd and since $u_c^p=1$, by putting $k=p$ in~\eqref{eq:commom} we get
$$1=[u_a,1]=[u_a,u_c^p]=\gamma^p \lambda ^{\frac{p(p-1)}{2}}u_b^p=\gamma^p.$$
We conclude that $\gamma$ is also a $p$-th root of unity.

Now, using the equalities 
$$u_au_cu_a^{-1}=\gamma u_bu_c \quad \text{and} \quad u_au_b=\delta u_bu_a$$ we compute
\begin{align*}
& [u_a^k,u_c]=u_a^{k-1}u_au_cu_a^{-1}u_a^{-(k-1)}u_c^{-1}= u_a^{k-1}(u_au_cu_a^{-1})u_a^{-(k-1)}u_c^{-1}= \\ \nonumber
& u_a^{k-1}(\gamma u_bu_c)u_a^{-(k-1)}u_c^{-1}= \gamma \delta ^{k-1}u_b u_a^{k-1} u_cu_a^{-(k-1)}u_c^{-1}=\gamma \delta ^{k-1}u_b[u_a^{k-1},u_c].
\end{align*}
Therefore, by induction 
\begin{equation}\label{eq:commom2}
[u_a^k,u_c]=\gamma^k \delta ^{(\sum_{i=1}^k i-1)}u_b^k=\gamma^k \delta ^{\frac{k(k-1)}{2}}u_b^k.
\end{equation}
Consequently, since $p$ is odd and since $u_b^p=1$, by putting $k=p$ in~\eqref{eq:commom2} we get
\begin{equation}\label{eq:apisfreg2}
[u_a^p,u_c]=\gamma^p \delta ^{\frac{p(p-1)}{2}}u_b^p=1. 
\end{equation}
Hence, by~\eqref{eq:apisfreg} and~\eqref{eq:apisfreg2} we conclude that $a^p$ is $f$-regular which completes the proof.
\end{proof}
The groups $G_{(\text{xi})},G_{(\text{xii})},G_{(\text{xiii})}$ are not of central type.
In these groups the element $a$ has the property that $C_G(a)=\langle a \rangle $ and therefore
by Lemma~\ref{lemma:trivcentralizer} these groups are not of central type.

Next, we show that the group $G_{(\text{xiv})}$ is of central type.
\begin{theorem}
Let $G=\langle a,b,c,d \rangle$ be a group of order $p^4$ such that 
$$a^p=b^p=c^p=d^p=1,\quad [a,b]=[a,c]=[a,d]=[b,c]=[b,d]=1, \quad [c,d]=a.$$
Then $G$ is a group of central type. 
\end{theorem}
\begin{proof}
First, notice that 
$$Z(G)=\langle a,b \rangle, C_G(c)= \langle a,b,c \rangle, C_G(d)= \langle a,b,d \rangle.$$
We wish to construct a nondegenerate cohomology class $[f]\in H^2(G,\C^*)$.
This twisted group algebra $\mathbb{C}^fG$ is generated as algebra by $u_a,u_b,u_c,u_d$ such that
\begin{equation}\label{eq:basicrelations}
u_a^p=u_b^p=u_c^p=u_d^p=1.
\end{equation}
We use a construction of crossed product in the following way.
Consider the following $C_p=\langle c \rangle$-action on the commutative group algebra $\mathbb{C}Z(G)$.
\begin{align*}
\psi _c: & \mathbb{C}Z(G)  \rightarrow  \mathbb{C}Z(G)\\
& u_a\mapsto \zeta _p u_a\\
&          u_b\mapsto u_b
\end{align*}
where $\zeta _p$ is a $p$-th root of unity.
Clearly, $\psi _c$ is an automorphism of order $p$ of the ring $\mathbb{C}Z(G)$ and therefore
the ring $R_1$ generated by $u_a,u_b,u_c$ is a crossed product of $C_p=\langle u_c \rangle$ over $\mathbb{C}Z(G)$.
In particular it is an associative algebra.
Consider the following $C_p=\langle d \rangle$-action on $R_1$.
\begin{align*}
\psi _d: & R_1\rightarrow R_1\\
& u_a\mapsto u_a\\
&          u_b\mapsto \zeta _p u_b\\
 &          u_c\mapsto u_a u_c.
\end{align*}
We will show that $\psi _d$ is an automorphism of order $p$ of the ring $R_1$.
We need to check if the following equalities hold:
\begin{enumerate}
\item  $\psi _d(u_b u_a u_b^{-1})= \psi _d(u_b)\psi_d( u_a)\psi_d( u_b)^{-1}.$
\item  $ \psi _d(u_c u_a u_c^{-1})= \psi _d(u_c)\psi_d( u_a)\psi_d( u_c)^{-1}.$
\item   $\psi _d(u_c u_b u_c^{-1})= \psi _d(u_c)\psi_d( u_b)\psi_d( u_c)^{-1}.$
\end{enumerate}
Indeed,
\begin{enumerate}
\item $$\psi _d(u_b u_a u_b^{-1})=\psi _d(u_a)= u_a,$$
\text{ and on the other hand, }
$$\psi _d(u_b)\psi_d( u_a)\psi_d( u_b)^{-1}=(\zeta _p u_b) u_a(\zeta _p u_b)^{-1}=u_a.$$
\item $$\psi _d(u_c u_a u_c^{-1})=\psi _d(\zeta _p^{-1}u_a)=\zeta _p^{-1}u_a,$$ 
\text{ and on the other hand, }
$$\psi _d(u_c)\psi_d( u_a)\psi_d( u_c)^{-1}=u_au_c u_a(u_au_c)^{-1}=u_a(u_c u_a u_c^{-1})u_a^{-1}=\\
u_a\zeta _p^{-1}u_a u_a^{-1}=\zeta _p^{-1}u_a.$$
\item $$\psi _d(u_c u_b u_c^{-1})=\psi_d(u_b)=\zeta _p u_b,$$
\text{ and on the other hand, }
$$\psi _d(u_c)\psi_d( u_b)\psi_d( u_c)^{-1}=(u_au_c) \zeta _p u_b (u_au_c)^{-1}=\zeta _p u_a(u_c u_b u_c^{-1})u_a^{-1}
=\zeta _p u_a u_b u_a^{-1}=\zeta _p u_b.$$
\end{enumerate}

We now need to check that $\psi _d$ has order $p$.
Clearly, 
$$\psi _d^p(u_a)=u_a \text{ and } \psi _d^p(u_b)=u_b.$$
Notice that
$$\psi_d^k(u_c)=\psi_d^{k-1}(u_au_c)=\psi_d^{k-1}(u_a)\psi_d^{k-1}(u_c) =u_a\psi_d^{k-1}(u_c)=\ldots =u_a^ku_c.$$
Therefore,
$$\psi_d^p(u_c)=u_a^p u_c=u_c.$$

Since $\psi _d$ is an automorphism of the ring $R_1$ we conclude that the ring $R$ generated by $u_a,u_b,u_c,u_d$ is a crossed product of 
$C_p=\langle u_d \rangle$ over $R_1$.
In particular it is an associative algebra. Hence the ring $R$ is an associative complex algebra generated by $u_a,u_b,u_c,u_d$ with the following relations.
$$ u_a^p=u_b^p=u_c^p=u_d^p=1$$
$$[u_a,u_b]=[u_a,u_d]=[u_b,u_c]=1,[u_a,u_c]=\zeta _p,[u_b,u_d]=\zeta _p, [u_c,u_d]=u_a.$$
The associativity of $R$ ensures that there exists a cocycle $f\in Z^2(G,\mathbb{C}^*)$ which satisfies the above relations.
We show that $f$ is nondegenerate by showing that there are no non-trivial $f$-regular elements in $G$.

Assume that $x=a^{r_1}b^{r_2}c^{r_3}d^{r_4}\in G$ is an $f$-regular element.
If $r_4\neq 0$ then $[u_b,u_x]\neq 1$ while $b \in C_G(x)$ and hence we can assume that $x=a^{r_1}b^{r_2}c^{r_3}$.\\
Then, if $r_3\neq 0$ then $[u_a,u_x]\neq 1$ while $a \in C_G(x)$ and hence we can assume that $x=a^{r_1}b^{r_2}$.\\
Then, if $r_2\neq 0$ then $[u_d,u_x]\neq 1$ while $d \in C_G(x)$ and hence we can assume that $x=a^{r_1}$.\\
But then, if $r_1\neq 0$ then $[u_c,u_x]\neq 1$ while $c \in C_G(x)$ and hence we get that $x$ is trivial and therefore $f$ is nondegenerate.
\end{proof}
Next, we show that the group $G_{(\text{xv})}$ is of central type in the $p>3$ case and in the $p=3$ case.
\begin{theorem}\label{th:existp41}
For prime $p\geq 5$, let $G$ be a group of order $p^4$ generated by four elements $\{a,b,c,d \}$ with the following relations, 
$$a^p=b^p=c^p=d^p=1, [a,b]=[a,c]=[a,d]=[b,c]=1,[d,b]=a,[d,c]=b.$$
And for $p=3$, $G_{15}$ let $G$ be a group of order $81$ generated by three elements $\{a,b,c \}$ with the following relations, 
$$a^9=b^3=c^3=1, [a,b]=1,cac^{-1}=ab,cbc^{-1}=a^6b.$$
Then, the groups $G$ are of central type.
\end{theorem}
\begin{proof}
We start with the case where $p\geq 5$.
First, notice that $Z(G)=\langle a \rangle$ and $G^{\shortmid}=\langle a,b \rangle$.
We wish to construct a nondegenerate cohomology class $[f]\in H^2(G,\C^*)$ such that the restriction of $[f]$ to $G^{\shortmid}$ is trivial.
This twisted group algebra $\mathbb{C}^fG$ is generated as algebra by $u_a,u_b,u_c,u_d$ such that
$$u_a^p=u_b^p=u_c^p=u_d^p=1.$$
We use a construction of crossed product in the following way.
First we define a cohomology class $[\alpha]$ over the subgroup $H:=\langle a,b,c \rangle$ using the following relations.
$$[\alpha ]:\quad [u_a,u_b]=[u_a,u_c]=1,[u_c,u_b]=\zeta _p,$$
where $\zeta _p$ is a $p$-th root of unit.
Consider the following $C_p=\langle d \rangle$-action on $\mathbb{C}^{\alpha}H$.
\begin{align*}
\psi _d: & \mathbb{C}^{\alpha}H\rightarrow \mathbb{C}^{\alpha}H\\
& u_a\mapsto \zeta _p u_a\\
&          u_b\mapsto u_au_b\\
 &          u_c\mapsto u_bu_c.
\end{align*}
We will show that $\psi _d$ is an automorphism of order $p$ of the ring $\mathbb{C}^{\alpha}H$.
We need to check if the following equalities hold.
\begin{enumerate}
\item  $\psi _d(u_b u_a u_b^{-1})= \psi _d(u_b)\psi_d( u_a)\psi_d( u_b)^{-1}.$
\item  $ \psi _d(u_c u_a u_c^{-1})= \psi _d(u_c)\psi_d( u_a)\psi_d( u_c)^{-1}.$
\item   $\psi _d(u_c u_b u_c^{-1})= \psi _d(u_c)\psi_d( u_b)\psi_d( u_c)^{-1}.$
\end{enumerate}
Indeed,
\begin{enumerate}
\item $\psi _d(u_b u_a u_b^{-1})=\psi _d(u_a)=\zeta _p u_a,$ \\
\text{ and on the other hand, }\\
$\psi _d(u_b)\psi_d( u_a)\psi_d( u_b)^{-1}=u_au_b \zeta _p u_a(u_au_b)^{-1}=\zeta _p u_a.$\\
\item $\psi _d(u_c u_a u_c^{-1})=\psi _d(u_a)=\zeta _p u_a,$ \\
\text{ and on the other hand, }\\
$\psi _d(u_c)\psi_d( u_a)\psi_d( u_c)^{-1}=u_bu_c \zeta _p u_a(u_bu_c)^{-1}=\zeta _p u_a.$\\
\item $\psi _d(u_c u_b u_c^{-1})=\psi_d(\zeta _p u_b)=\zeta _p u_au_b,$\\
\text{ and on the other hand, }\\
$\psi _d(u_c)\psi_d( u_b)\psi_d( u_c)^{-1}=u_bu_c u_au_b (u_bu_c)^{-1}=u_au_b[u_c,u_b]=\zeta _p u_au_b.$
\end{enumerate}
We now need to check that $\psi _d$ has order $p$.
Clearly 
$$\psi_d^p(u_a)=\zeta _p^p u_a=u_a.$$
Notice that
\begin{align*}
& \psi_d^k(u_b)=\psi_d^{k-1}(u_au_b)=\psi_d^{k-1}(u_a)\psi_d^{k-1}(u_b)=\ldots \\
& =\bigg(\prod _{i=0}^{k-1}\psi_d^i(u_a)\bigg)u_b=\bigg(\prod _{i=0}^{k-1}\zeta _p^i u_a\bigg)u_b=\zeta _p^{\frac{k(k-1)}{2}}u_a^k u_b.
\end{align*}
Therefore,
$$\psi_d^p(u_b)=\zeta _p^{\frac{p(p-1)}{2}}u_a^p u_b=u_b.$$
In order to show that $ \psi_d^p(u_c)=u_c$ we will use two observations.
First, it is easy to prove by induction that for any natural $n\geq 1$,
$$\frac{(n+1)n}{2}+\frac{n(n-1)}{2}+\ldots +\frac{2\cdot 1}{2}=\frac{(n+2)(n+1)n}{6}.$$
Second, for any prime number $p\geq 5$
$$\frac{p(p-1)(p-2)}{6}$$
is divisible by $p$. We will also use that $u_a$ and $u_b$ commute. Now,
\begin{align*}
& \psi_d^p(u_c)=\psi_d^{p-1}(u_bu_c)=\psi_d^{p-1}(u_b)\psi_d^{p-1}(u_c)=\ldots =\bigg(\prod _{i=0}^{p-1}\psi_d^i(u_b)\bigg)u_c=\\
& \bigg(\prod _{i=0}^{p-1} \zeta _p^{\frac{i(i-1)}{2}}u_a^i u_b\bigg)u_c=\bigg(\zeta _p^{\sum_{i=0}^{p-1}\frac{i(i-1)}{2}}u_a^{\sum_{i=0}^{p-1} i} u_b^p \bigg)u_c=
\zeta _p^{\frac{p(p-1)(p-2)}{6}}u_a^{\frac{p(p-1)}{2}}=u_c.
\end{align*}
Since $\psi _d$ is an automorphism of the ring $\mathbb{C}^{\alpha}H$ we conclude that the ring $R$ generated by $u_a,u_b,u_c,u_d$ is a crossed product of $C_p$ over $\mathbb{C}^{\alpha}H$.
In particular it is an associative algebra. Hence the ring $R$ is an associative complex algebra generated by $u_a,u_b,u_c,u_d$ with the following relations.
\begin{equation}
 [u_a,u_b]=[u_a,u_c]=1,[u_d,u_a]=\zeta _p,[u_c,u_b]=\zeta _p, [u_d,u_b]=u_a,[u_d,u_c]=u_b.
\end{equation}
The associativity of $R$ ensures that there exists a cocycle $f\in Z^2(G,\mathbb{C}^*)$ which satisfies the above relations.
We show that $f$ is nondegenerate by showing that there are no non-trivial $f$-regular elements in $G$.
Assume that $x=a^{r_1}b^{r_2}c^{r_3}d^{r_4}\in G$ is an $f$-regular element.
If $r_4\neq 0$ then $[u_a,u_x]\neq 1$ while $a \in C_G(x)$ and hence we can assume that $x=a^{r_1}b^{r_2}c^{r_3}$.
Then, if $r_3\neq 0$ then $[u_b,u_x]\neq 1$ while $b \in C_G(x)$ and hence we can assume that $x=a^{r_1}b^{r_2}$.
Then, if $r_2\neq 0$ then $[u_c,u_x]\neq 1$ while $c \in C_G(x)$ and hence we can assume that $x=a^{r_1}$.
But then, if $r_1\neq 0$ then $[u_d,u_x]\neq 1$ while $d \in C_G(x)$ and hence we get that $x$ is trivial and therefore $f$ is nondegenerate.

The case $p=3$ is done in a similar way.
We start by defining a cohomology class $[\beta ]$ over the commutative group $H=\langle a,b \rangle$ as follows.
$$[\beta ]:\quad [u_a,u_b]=\zeta _3,$$
where $\zeta _3$ is a root of unity of order $3$.
Then we define a $\langle c \rangle$-action on $\mathbb{C}^{\beta}H$ by 
\begin{align*}
\psi _{c}: & \mathbb{C}^{\alpha}H\rightarrow \mathbb{C}^{\alpha}H\\
& u_a\mapsto \zeta _9 u_bu_a\\
&          u_b\mapsto u_a^6 u_b,
\end{align*}
where $\zeta _9$ is a root of unity of order $9$.
We check that indeed $\psi _{c}$ is an automorphism of order $3$ of the ring $\mathbb{C}^{\beta}H$ and consequently we get an associative ring $R$
spanned by $u_a,u_b,u_c$ with the following relations
$$u_a^9=u_b^3=u_c^3=1,[u_a,u_b]=\zeta _3, [u_c,u_a]=\zeta _9u_b,[u_c,u_b]=u_a^6.$$
The associativity of $R$ ensures that there exists a cocycle $f\in Z^2(G,\mathbb{C}^*)$ which satisfies the above relations.
We show that $f$ is nondegenerate by showing that there are no non-trivial $f$-regular elements in $G$. Notice that 
$$C_G(a)=C_G(b)=\langle a,b \rangle, C_G(c)=\langle c,a^3 \rangle.$$
Assume that $x=a^ib^jc^k \in G$ is an $f$-regular element. 
If $k\neq 0$ then $[u_{a^3},u_x]\neq 1$ while $a^3 \in C_G(x)$ and hence we can assume that $x=a^i b^j$.
If $j\neq 0$ then $[u_a,u_x]\neq 1$ while $a \in C_G(x)$ and hence we can assume that $x=a^i$.
But then, if $i\neq 0$ then $[u_b,u_x]\neq 1$ while $b \in C_G(x)$ and hence we get that $x$ is trivial and therefore $f$ is nondegenerate.
\end{proof}
This completes the proof of Theorem 1.
\subsection{$p$ is even} 
In this section we complete the proof of Theorem 2.
All the numbering in this section are with correspondence with Table~\ref{tab:16}.
Recall that we already proved that from the groups  $G_{(\text{i})}$-$G_{(\text{vii})}$, the only groups of central type are
 $G_{(\text{iii})}, G_{(\text{v})}$.

We start this section by showing that the group $G_{(\text{viii})}$ is not of central type.
\begin{lemma}\label{th:group282}
Let $G=\langle a,b \rangle$ be a group of order $16$ such that 
$$a^4=b^4=1,\quad [a,b]=a^2.$$
Then, for any cocycle $f\in Z^2(G,\mathbb{C}^*)$, the element $b^2$ is $f$-regular.
In particular, $G$ is not a group of central type.
\end{lemma}
\begin{proof}
Let $f\in Z^2(G,\mathbb{C}^*)$, and assume that 
$$[u_a,u_b]=\delta u_a^2.$$
Then
\begin{equation}\label{eq:maybelast}
u_au_bu_a^{-1}=\delta u_a^2u_b.
\end{equation}
Since $a^2$ and $b$ commutes in $G$, then by Lemma~\ref{lemma:commupto}, $\delta$ is a root of unity of order dividing $4$.
Therefore, using~\eqref{eq:maybelast} we get
\begin{equation}
u_au_b^2u_a^{-1}=(\delta u_a^2u_b)^2=\delta ^2u_a^2(u_bu_a^2)u_b=\delta ^2u_a^2(\delta ^2u_a^2u_b)u_b=\delta ^4=1.
\end{equation} 
Therefore, $b^2$ is an $f$-regular element in $G$. Moreover, $u_b$ is a nontrivial central element in $\C ^fG$ and hence the group $G$ is not of central type
\end{proof}

Next, we show that the group $G_{(\text{xi})}$ is of central type.
\begin{theorem}\label{th:group9}
Let $G=\langle a,b,c \rangle$ be a group of order $16$ such that 
$$a^4=b^2=c^2=1,\quad [a,b]=[b,c]=1, \quad [a,c]=a^2.$$
Then $G$ is a group of central type. 
\end{theorem}
\begin{proof}
We wish to construct a nondegenerate cohomology class $[f]\in H^2(G,\C^*)$.
The corresponding twisted group algebra $\mathbb{C}^fG$ is generated as algebra by $u_a,u_b,u_c$ such that
\begin{equation}
u_a^4=u_b^2=u_c^2=1.
\end{equation}
We use a construction of crossed product in the following way.
We start from the twisted group algebra $R_1=\C ^{\alpha}\langle a,b \rangle$, where $[\alpha] \in H^2(\langle a,b \rangle,\C^*)$
is defined by the relation
$$[u_a,u_b]=-1.$$

Consider the following $C_2=\langle c \rangle$-action on $R_1$.
\begin{align*}
\psi _c: & R_1  \rightarrow  R_1\\
& u_a\mapsto i u_a^3\\
& u_b\mapsto u_b.
\end{align*}
By setting $\psi_c(u_a^j)=\psi _c(u_a)^j$ we ensure that $\psi _c$ is a homomorphism of $R_1$.

Next, we check that the automorphism $\psi _c$ has order $2$.
Clearly $\psi _c ^2(u_b)=u_b$.
Now,
\begin{equation}
\psi _c ^2(u_a)=\psi _c (iu_a ^3)=i\psi _c(u_a)^3=i\cdot i^3 u_a^9=u_a.
\end{equation}
We conclude that indeed $\psi _c$ is an automorphism of order $2$ of $R_1$.

Since $\psi _c$ is an automorphism of the ring $R_1$, the ring $R$ generated by $u_a,u_b,u_c$ is a crossed product of 
$C_2=\langle u_c \rangle$ over $R_1$.
In particular it is an associative algebra. Hence the ring $R$ is an associative complex algebra generated by $u_a,u_b,u_c$ with the following relations.
$$ u_a^4=u_b^2=u_c^2=1,\quad  [u_a,u_b]=-1,[u_c,u_b]=1,[u_c,u_a]=iu_a^2.$$
The associativity of $R$ ensures that there exists a cocycle $f\in Z^2(G,\mathbb{C}^*)$ which satisfies the above relations.
We show that $f$ is nondegenerate by showing that there are no non-trivial $f$-regular elements in $G$.

Let $x=a^tb^jc^k \in G$ be an $f$-regular element.
If $k\neq 0$ then $[u_x,u_a^2]\neq 1$ while $a^2\in C_G(x)$. Hence we may assume that $k= 0$ and  $x=a^tb^j$.\\
If $j\neq 0$ then $[u_x,u_a]\neq 1$ while $a\in C_G(x)$. Hence we may assume that $j=0$ and $x=a^t$.\\
If $t$ is odd, then $[u_x,u_b]\neq 1$ while $b\in C_G(x)$. Hence we may assume that $t$ is even.\\
But then, if $t\equiv 2($mod $4)$, then $[u_x,u_c]\neq 1$ while $c\in C_G(x)$ and therefore $x$ is trivial.
We conclude that $f$ is nondegenerate and $G$ is a group of central type. 
\end{proof}

Next, we show that the group $G_{(\text{x})}$ is of central type.
\begin{theorem}\label{th:group10}
Let $G=\langle a,b,c \rangle$ be a group of order $16$ such that 
$$a^4=b^2=c^2=1,\quad [a,b]=[b,c]=1, \quad [a,c]=b.$$
Then $G$ is a group of central type. 
\end{theorem}
\begin{proof}
We wish to construct a nondegenerate cohomology class $[f]\in H^2(G,\C^*)$.
The corresponding twisted group algebra $\mathbb{C}^fG$ is generated as algebra by $u_a,u_b,u_c$ such that
\begin{equation}
u_a^4=u_b^2=u_c^2=1.
\end{equation}
We use a construction of crossed product in the following way.
We start from the twisted group algebra $R_1=\C ^{\alpha}\langle a,b \rangle$, where $[\alpha] \in H^2(\langle a,b \rangle,\C^*)$
is defined by the relation
$$[\alpha]:\quad [u_a,u_b]=-1.$$
Consider the following $C_2=\langle c \rangle$-action on $R_1$.
\begin{align*}
\psi _c: & R_1  \rightarrow  R_1\\
& u_a\mapsto - u_bu_a\\
& u_b\mapsto u_b,
\end{align*}
and $\psi_c(u_a^j)=\psi _c(u_a)^j$, $\psi_c(u_b^j)=\psi _c(u_b)^j$.

We will show that $\psi _c$ is an automorphism of order $2$ of the ring $R_1$.
We need to check if the following equality holds:
$$\psi _c(u_a u_b u_a^{-1})= \psi _c(u_a)\psi_c( u_b)\psi_c( u_a)^{-1}.$$ 
Indeed,
$$\psi _c(u_a u_b u_a^{-1})=\psi _c(-u_b)=-u_b,$$
and on the other hand
$$\psi _c(u_a)\psi_c( u_b)\psi_c( u_a)^{-1}=(- u_b u_a)(u_b)(-u_a^3u_b)=-u_b.$$

Next, we check that the automorphism $\psi _c$ has order $2$.
Clearly $\psi _c ^2(u_b)=u_b$.
Now,
\begin{equation}
\psi _c ^2(u_a)=\psi _c (- u_b u_a)=-\psi _c(u_a)\psi _c(u_b)=-(-u_au_b)(u_b)=u_a.
\end{equation}
We conclude that indeed $\psi _c$ is an automorphism of order $2$ of $R_1$.

Since $\psi _c$ is an automorphism of the ring $R_1$, the ring $R$ generated by $u_a,u_b,u_c$ is a crossed product of 
$C_2=\langle u_c \rangle$ over $R_1$.
In particular it is an associative algebra. Hence the ring $R$ is an associative complex algebra generated by $u_a,u_b,u_c$ with the following relations.
$$ u_a^4=u_b^2=u_c^2=1,\quad  [u_a,u_b]=-1,[u_c,u_b]=1,[u_c,u_a]=-u_b.$$
The associativity of $R$ ensures that there exists a cocycle $f\in Z^2(G,\mathbb{C}^*)$ which satisfies the above relations.
We show that $f$ is nondegenerate by showing that there are no non-trivial $f$-regular elements in $G$.

Let $x=a^tb^jc^k \in G$ be an $f$-regular element.
If $k\neq 0$ then $[u_x,u_a^2]\neq 1$ while $a^2\in C_G(x)$. Hence we may assume that $k= 0$ and  $x=a^tb^j$.\\
If $j\neq 0$ then $[u_x,u_a]\neq 1$ while $a\in C_G(x)$. Hence we may assume that $j=0$ and $x=a^t$.\\
If $t$ is odd, then $[u_x,u_b]\neq 1$ while $b\in C_G(x)$. Hence we may assume that $t$ is even.\\
But then, if $t\equiv 2($mod $4)$, then $[u_x,u_c]\neq 1$ while $c\in C_G(x)$ and therefore $x$ is trivial.
We conclude that $f$ is nondegenerate and $G$ is a group of central type. 
\end{proof}
Next, we show that the group $G_{(\text{xi})}$ is not of central type.
\begin{lemma}
 Let $G=\langle a,b,c \rangle$ be a group of order $16$ such that
 $$a^4=c^2=1,\quad a^2=b^2, \quad [a,c]=[b,c]=1,[a,b]=a^2.$$
 Then, for any cocycle $f\in Z^2(G,\mathbb{C}^*)$, the element $a^2$ is $f$-regular.
 In particular, $G$ is not a group of central type.
\end{lemma}
\begin{proof}
Let $f\in Z^2(G,\mathbb{C}^*)$ and assume that $[u_a,u_c]=\delta$.
By Lemma~\ref{lemma:commupto}, $\delta  \in \{1,-1\}$.
In particular, 
\begin{equation}\label{eq:2group11}
[u_{a^2},u_c]=1. 
\end{equation}
Hence, in order to show that $a^2$ is $f$-regular, we need to show that $[u_{a^2},u_b]=1$.\\
Since $[a,b]=a^2$, there exists a scalar $\gamma$ such that
$$[u_a,u_b]=\gamma u_a^2,$$
or equivalently, $u_bu_a^{-1}u_b^{-1}=\gamma u_a$.
Therefore,
\begin{equation}
u_a^2=u_b^2=u_b^{-2}=u_bu_b^{-2}u_b^{-1}=u_bu_a^{-2}u_b^{-1}=(u_bu_a^{-1}u_b^{-1})^2=\gamma ^2u_a^2. 
\end{equation}
Consequently, $\gamma  \in \{1,-1\}$.
Hence, 
\begin{equation}\label{eq:group1122}
[u_a^2,u_b]=u_a^2(u_bu_a^{-2}u_b^{-1})= u_a^2(u_bu_a^{-1}u_b^{-1}u_b u_a^{-1}u_b^{-1})=u_a^2(u_bu_a^{-1}u_b^{-1})^2=1.
\end{equation}
Therefore, by~\eqref{eq:2group11} and~\eqref{eq:group1122}, $a^2$ is $f$-regular and hence $f$ is not nondegenerate.
\end{proof}
Finally, the groups $G_{(\text{xii})},G_{(\text{xiii})},G_{(\text{xiv})}$ admit a cyclic subgroup of order $8$ and hence by
Lemma~\ref{lemma':toobycyc} these groups are not of central type..

This completes the proof of Theorem 2.
\section{The semi-center of simple twisted group algebra}\label{scsc}
We use the fact that the partition of a
group $G$ to conjugacy classes refines the partition of $G$ to
cosets of the commutator $G^{\shortmid}$.
Let $f\in Z^2(G,\mathbb{C}^*)$ be a nondegenerate cocycle. By~\eqref{eq:scsimp}, for every
$\lambda \in \hat{G}$, dim$((\mathbb{C}^fG) _{\lambda})=1$. Hence,
the support of any element in $(\mathbb{C}^fG) _{\lambda}$ is a
unique conjugacy class of $G$, say $[x]$.
Define the following group homomorphism.
\begin{equation}\label{eq:longer}
\begin{aligned}
\varphi=\varphi(f) :\hat{G} & \rightarrow G/G^{\shortmid}\\
\lambda & \mapsto xG^{\shortmid}.
\end{aligned}
\end{equation}
Let $\lambda _0\in \hat{G}$.
By \cite[(25)]{ginosar2012semi} it can easily be verified that $\lambda _0(x)=1$ for every $x\in \Gamma(f)$ (see Corollary~\ref{th:theorem1}) if and only if $\lambda _0\in$ker$(\varphi)$.
Therefore,
\begin{equation}\label{eq:nimas}
\mathbb{C}^{\hat{f}}\text{ker}(\varphi)=Z(\text{Sz}(\mathbb{C}^fG)).
\end{equation}
Next, we prove Theorem 3.
We show that if $\C ^fG$ is a simple twisted group algebra and the restriction of $f$ to $G^{\shortmid}$, which we denote by $f^{\shortmid}$, is also
nondegenerate then the semi-center of $\C ^fG$ is simple.
In particular, if $G^{\shortmid}$ is a Hall subgroup of
$G$, then the semi-center of $\C ^fG$ is simple.\\
\centerline{\textbf{Proof of Theorem 3.}}
First, we show that if there are no non-trivial $(\lambda,f)$-regular conjugacy classes contained in $G^{\shortmid}$
then the semi-center of $\mathbb{C}^fG$ is simple.
Notice that in this case, the associated kernel of $\varphi$ (see~\eqref{eq:longer}) is trivial. Hence, by~\eqref{eq:nimas} we conclude
that the center of Sz$(\mathbb{C}^fG)$ is one dimensional
and as a consequence the semi-center of $(\mathbb{C}^fG)$ is simple.

Next, we show that there are no non-trivial $(\lambda,f)$-regular conjugacy class contained in $G^{\shortmid}$.
Assume that $[g]\subseteq G^{\shortmid}$ is a $(\lambda,f)$-regular conjugacy class.
Let $x\in C_{G^{\shortmid}}(g)$, in particular $x\in G^{\shortmid}$. Therefore,
$\delta (x)=1$ for any $\delta \in \hat{G}$. Since $g$ is a $(\lambda,f)$-regular element in $G$,
$$\alpha _f(g,x)=\alpha _{f^{\shortmid}}(g,x)=\lambda (x)=1.$$
Hence, $g$ is an $f^{\shortmid}$-regular element.
By the nondegeneracy of $f^{\shortmid}$ we conclude that $g=1$. Therefore,
there are no non-trivial $(\lambda,f)$-regular conjugacy class contained in $G^{\shortmid}$.
Hence, by the first part the semi-center of  is simple.
By \cite[Corollary 4]{DeMeyer} the restriction of a nondegenerate cocycle to any Hall subgroup is also nondegenerate.
Hence, by the first part, if $G^{\shortmid}$ is a Hall subgroup of
$G$, then the semi-center of $\C ^fG$ is simple.\qed
\begin{remark}\label{remark:sf}
Let $G$ be a group of central type of order $n^2$ where $n$ is square-free number.
In \cite{ginosargradings} we show that in this case $G^{\shortmid}$ is a Hall subgroup of
$G$. Hence, by Theorem 3 the semi-center of $\C ^fG$ is simple.
\end{remark}
Corollaries~\ref{cor:Htriv},~\ref{cor:Hall} deal with two extremes
of ker$(\varphi)$ (see~\eqref{eq:longer}). Either ker$(\varphi)$ is trivial or ker$(\varphi)=\hat{G}$.
These extremes correspond to the two extremes presented in Problem 1
for nondegenerate cocycles.
\begin{corollary}\label{cor:Htriv}
Let $f\in Z^2(G,\mathbb{C}^*)$ be a nondegenerate cocycle. With the above notations the following are equivalent.
\begin{enumerate}
\item ker$(\varphi)$ is trivial. In other words, there are no non-trivial $(\lambda,f)$-regular conjugacy
classes contained in $G^{\shortmid}$.
\item $\varphi$ is an isomorphism.
\item Sz$(\mathbb{C}^fG)=\mathbb{C}^{\hat{f}}\hat{G}$ is simple. In other words, $\hat{f}\in Z^2(\hat{G},\mathbb{C}^*)$
is a nondegenerate cocycle, and in particular, $\hat{G}$ is a group of central type.
\end{enumerate}
\end{corollary}
\begin{proof}
Since $\hat{G}\cong G/G^{\shortmid}$, ker$(\varphi)$ is trivial if and only if
$\varphi$ is an isomorphism. By equation~\eqref{eq:nimas},
ker$(\varphi)$ is trivial if and only if the center of
$\mathbb{C}^{\hat{f}}\hat{G}$ is one dimensional and hence
$\hat{f}$ is nondegenerate cocycle.
\end{proof}
The other extreme is giving as follows.
\begin{corollary}\label{cor:Hall}
Let $f\in Z^2(G,\mathbb{C}^*)$ be a nondegenerate cocycle. With the above notations the following are equivalent.
\begin{enumerate}
\item ker$(\varphi)=\hat{G}$. In other words, all the $(\lambda,f)$-regular conjugacy
classes are contained in $G^{\shortmid}$.
\item $\hat{f}$ is cohomologically trivial.
\item Sz$(\mathbb{C}^fG)=\mathbb{C}^{\hat{f}}\hat{G}$ is commutative.
\item With the notations of Corollary~\ref{th:theorem1},
$$ \sum _{x\in \Gamma(f)\cap G^{\shortmid}}[G:G' C_G(x)]=|\hat{G}|.$$
\end{enumerate}
Moreover, if these conditions hold then $Z(G)\leq G^{\shortmid}$.
\end{corollary}
\begin{proof}
We first show that the second and the third conditions are equivalent.
Assume that $\mathbb{C}^{\hat{f}}\hat{G}$ is commutative. Then, since the order of the cohomology class $[f]\in H^2(G,\C ^*)$ divides
the dimension of each projective $f$-representation (see \cite[p.406, Remark]{GA})
$\hat{f}$ is cohomologically trivial.
On the other hand, since $\hat{G}$ is commutative, if $\hat{f}$ is cohomologically trivial, then $\mathbb{C}^{\hat{f}}\hat{G}$ is commutative.
Next, we show that the first and the third conditions are equivalent.
Since $f$ is nondegenerate, the dimension of Sz$(\mathbb{C}^fG)$ is $|\hat{G}|$.
Therefore, by~\eqref{eq:nimas}, 
Sz$(\mathbb{C}^fG)$ is commutative if and only if ker$(\varphi)=\hat{G}$.
The equivalence between the first and the fourth conditions is clear from Corollary~\ref{th:theorem1}.
Finally, since any central element is $(\lambda,f)$-regular for some $\lambda \in \hat{G}$ (see Lemma~\ref{remark:sc}), the last claim is clear
from the first condition.
\end{proof}
By Lemma~\ref{remark:sc}, nondegenerate cocycles $f\in Z^2(G,\C ^*)$ over abelian groups $G$
are natural examples for affirmative answer to Problem 1(ii). However, there exist more sophisticated examples.
\begin{corollary}\label{cor:mat}
Let $f\in H^2(G,\mathbb{C}^*)$ be a nondegenerate cocycle, assume
also
 $Z(G)\not \subseteq G^{\shortmid}$.
If $\hat{G}\cong C_p\times C_p$ then Sz$(\mathbb{C}^fG)\cong
M_p(\mathbb{C})$.
\end{corollary}
\begin{proof}
By~\eqref{eq:scsimp},
$$\text{Sz}(\mathbb{C}^fG)=\mathbb{C}^{\hat{f}}(C_p\times C_p).$$
Now, by Corollary~\ref{cor:Hall}, since $Z(G)\not \subseteq G^{\shortmid}$, $\mathbb{C}^{\hat{f}}(C_p\times C_p)$ is non-commutative and hence
by Corollary~\ref{cor:Hall}  $[\hat{f}]$ is non-trivial.
Since any cohomology class of $C_p\times C_p$ is either trivial or nondegenerate (see \cite[\S 2]{ginosargradings})
we get that $\hat{f}$ is nondegenerate. Consequently, $\mathbb{C}^{\hat{f}}(C_p\times C_p)$ is simple.
\end{proof}
In Theorem 5, the idea of the construction of the cohomology class $[f]$ is based on the following lemma.
\begin{lemma}\label{lemma:capG}
Let $f$ be nondegenerate cocycle and denote by $f^{\shortmid}$ the restriction of $f$ to $G^{\shortmid}$. Then
$$ Sz(\mathbb{C}^fG)\cap \mathbb{C}^{f^{\shortmid}}G^{\shortmid}\subseteq Z(\mathbb{C}^{f^{\shortmid}}G^{\shortmid}).$$
\end{lemma}
\begin{proof}
By Proposition~\ref{th:weightspaces}, it is enough to show that if $x\in G^{\shortmid}$ is a $(\lambda,f)$-regular element for some
$\lambda \in \hat{G}$, then it is also $f^{\shortmid}$-regular.
Assume that $x\in G^{\shortmid}$ is a $(\lambda,f)$-regular element.
Let $g\in C_{G^{\shortmid}}(x)$, in particular $g\in G^{\shortmid}$. Therefore,
$\delta (g)=1$ for any $\delta \in \hat{G}$. Since $x$ is $(\lambda,f)$-regular,
$$[u_x,u_g]=\lambda (x)=1.$$
Consequently $x$ is $f^{\shortmid}$-regular.
\end{proof}
Note that for any normal subgroup $N$ of a group $G$, a $G$-conjugacy class splits into $[G:NC_G(g)]$, $N$-conjugacy classes.
Let $x\in G^{\shortmid}$ be a $(\lambda ,f)$-regular element. Then, the $G$-conjugacy class $[x]$ splits into $[G:G^{\shortmid}C_G(x)]$ many $G^{\shortmid}$-conjugacy classes.
Consequently, with the notations of Corollary~\ref{th:theorem1},
\begin{equation}\label{eq:lastornot123}
\left| \bigcup _{\lambda \in \hat{G}} \{ [x]_{G^{\shortmid}}| x \text{ is }(\lambda ,f)\text{-regular}    \} \right| =\sum _{x\in \Gamma(f)\cap G^{\shortmid}}[G:G' C_G(x)].
\end{equation}
From~\eqref{eq:lastornot123} and Corollary~\ref{cor:Hall} we get
\begin{corollary}\label{lemma:smallcom}
Let $f\in Z^2(G,\mathbb{C}^*)$ be a nondegenerate cocycle. If the number of $G^{\shortmid}$ conjugacy classes (in $G^{\shortmid}$) is smaller than
the cardinality of $\hat{G}$, then Sz$(\mathbb{C}^fG)$ is not commutative.
In particular, if Sz$(\mathbb{C}^fG)$ is commutative, then $|G^{\shortmid}|\geq |G/G^{\shortmid}|$.
\end{corollary}
In the next section we will deal with the semi-center of simple twisted group algebras $\mathbb{C}^fG$ where $G$ is a $p$-group for a prime number $p$.
We will use the following proposition.
\begin{proposition}\label{prop:eq}
Let $f\in Z^2(G,\mathbb{C})$ be a nondegenerate cocycle and assume\\
that Sz$(\mathbb{C}^fG)$ is commutative.
Then, $|Z(G)|\leq \sqrt{|G|}$. In the specific case where $|Z(G)|= \sqrt{|G|}$, then $G^{\shortmid}=Z(G)\cong\hat{G}$.
\end{proposition}
\begin{proof}
By Corollary~\ref{cor:Hall}, since Sz$(\mathbb{C}^fG)$ is commutative then $Z(G)\subseteq G^{\shortmid}$. In particular, 
\begin{equation}\label{eq:smallcenter1}
|Z(G)|\leq |G^{\shortmid}|.
\end{equation}
On the other hand, by \cite[Theorem C]{ginosar2012semi}, $Z(G)$ embeds into $\hat{G}$ and therefore 
\begin{equation}\label{eq:smallcenter2}
|Z(G)|\leq |\hat{G}|=\frac{|G|}{|G^{\shortmid}|}.
\end{equation}
From~\eqref{eq:smallcenter1} and~\eqref{eq:smallcenter2} we get $|Z(G)|\leq \sqrt{|G|}$.
In the case where $|Z(G)|= \sqrt{|G|}$ we get equalities in both~\eqref{eq:smallcenter1} and~\eqref{eq:smallcenter2}.
Since $Z(G)$ and $G^{\shortmid}$ are of the same order and $Z(G) \subseteq G^{\shortmid}$ we have $G^{\shortmid}=Z(G)$. Since $Z(G)$ embeds into $\hat{G}$
and they have the same order we have $Z(G)\cong\hat{G}$.
\end{proof}
\section{Semi-center of simple twisted group algebras over $p$-groups.}
We start this section by proving that the semi-center of the simple twisted group algebras presented in Theorem~\ref{th:existp41}, which correspond to the groups $G_{(\text{xv})}$, is commutative.
Then we will show that $G_{(\text{xv})}$ is unique with this properties. In other words, we will show that if $G$ is a group of order $p^4$, where $p$ now is any prime, and if there exists a 
nondegenerate cocycle $f\in Z^2(G,\mathbb{C}^*)$ such that 
Sz$(\mathbb{C}^fG)$ is commutative, then $G\cong G_{(\text{xv})}$. In particular, there is no group $G$ of order $16$ which admits a 
nondegenerate cocycle $f\in Z^2(G,\mathbb{C}^*)$ such that 
Sz$(\mathbb{C}^fG)$ is commutative.

\begin{theorem}\label{th:existp4}
Let $G:=G_{(\text{xv})}$ for primes $p\geq 5$ and let $H:=G_{(\text{xv})}$ for $p=3$. The twisted group algebras which correspond to the following cohomology
classes $[f]\in H^2(G,\mathbb{C}^*)$ and $[\beta ]\in  H^2(H,\mathbb{C}^*)$ admits a commutative semi-center.
$$[f]:u_a^p=u_b^p=u_c^p=u_d^p=1  [u_a,u_b]=[u_a,u_c]=1,[u_d,u_a]=\zeta _p,[u_c,u_b]=\zeta _p, [u_d,u_b]=u_a,[u_d,u_c]=u_b.$$
$$[\beta ]:u_a^9=u_b^3=u_c^3=1,[u_a,u_b]=\zeta _3, [u_c,u_a]=\zeta _9u_b,[u_c,u_b]=u_a^6.$$
\end{theorem}
\begin{proof}
In the proof of Theorem~\ref{th:existp41} it was shown that $[f]$ and $[\beta ]$ are indeed cohomology classes.
We start by showing that Sz$(\mathbb{C}^fG)$ is commutative.
Since $G^{\shortmid}=\langle a,b \rangle$ we can identify
$$\hat{G}\cong C_p\times C_p=\langle \lambda _{c} \rangle \times \langle \lambda _{d} \rangle.$$
Here,
$$\lambda _c(a)=\lambda _c(b)=\lambda _c(d)=1 ,\lambda _c(c)=\zeta _p,$$
and
$$\lambda _d(a)=\lambda _d(b)=\lambda _d(c)=1, \lambda _d(d)=\zeta _p,$$
where $\zeta _p$ is a primitive $p$-th root of unity.
Since $a$ is central and
$$[u_a,u_c]=[u_a,u_b]=1,[u_d,u_a]=\zeta _p,$$
we get that $a$ is $(\lambda _d,f)$-regular.\\
Similarly, since $C_G(b)=\langle a,b,c \rangle$ and
$$[u_a,u_b]=1,[u_c,u_b]=\zeta _p,$$
we get that $b$ is $(\lambda _c,f)$-regular.
Consequently, with the notations of~\eqref{eq:longer},
$$\varphi (\lambda _c)=bG^{\shortmid}=G^{\shortmid} \text{ and } \varphi (\lambda _d)=aG^{\shortmid}=G^{\shortmid}.$$
Therefore, $\varphi$ is trivial and hence by Corollary~\ref{cor:Hall}, Sz$(\mathbb{C}^fG)$ is commutative. 

Next, we show that Sz$(\mathbb{C}^{\beta}H)$ is commutative.
Since $H^{\shortmid}=\langle a^3,b \rangle$ we can identify
$$\hat{H}\cong C_p\times C_p=\langle \lambda _{a} \rangle \times \langle \lambda _{c} \rangle.$$
Here,
$$\lambda _a(b)=\lambda _a(c)=1 ,\lambda _a(a)=\zeta _3,$$
and
$$\lambda _c(a)=\lambda _c(b)=1, \lambda _c(c)=\zeta _3,$$
where $\zeta _3$ is a root of unity of order $3$.
Since $a^3$ is central and
$$[u_{a^3},u_b]=1,[u_{a^3},u_c]=\zeta _3,$$
we get that $a^3$ is $(\lambda _c,\beta)$-regular.\\
Similarly, since $C_G(b)=\langle a,b \rangle$ and
$$[u_a,u_b]=\zeta _3,$$
we get that $b$ is $(\lambda _a,\beta)$-regular.
Consequently, with the notations of~\eqref{eq:longer},
$$\varphi (\lambda _a)=bG^{\shortmid}=G^{\shortmid} \text{ and } \varphi (\lambda _c)=a^3G^{\shortmid}=G^{\shortmid}.$$
Therefore, $\varphi$ is trivial and hence by Corollary~\ref{cor:Hall}, Sz$(\mathbb{C}^fG)$ is commutative. 
\end{proof}
Notice that $Z(G_{(\text{xv})})$ is a proper subgroup of $G_{(\text{xv})}^{\shortmid}$ of cardinality $p$.
It turns out that this is a necessary condition on a group $G$ of order $p^4$ in order to admit a nondegenerate cocycle $f\in Z^2(G,\mathbb{C})$ such that Sz$(\mathbb{C}^fG)$ is commutative.
\begin{lemma}\label{lemma:centproppp}
Let $G$ be a group of order $p^4$ where $p$ is prime.
Assume that there is a nondegenerate cocycle $f\in Z^2(G,\mathbb{C})$ such that Sz$(\mathbb{C}^fG)$ is commutative.
Then $Z(G)$ is a proper subgroup of $G^{\shortmid}$ of cardinality $p$.
\end{lemma}
\begin{proof}
The group $G$ admits a normal subgroup of order $p^i$ for any $0\leq i\leq 4$ (see \cite[Exercise 4.2]{rotman}). In particular,
$G$ admits a normal subgroup $N$ of order $p^2$. Since $|G/N|=p^2$ we conclude that $G/N$ is a abelian.
Therefore $|G^{\shortmid}|\leq p^2$. If $|G^{\shortmid}|< p^2$, then by Corollary~\ref{lemma:smallcom},
Sz$(\mathbb{C}^fG)$ is non-commutative which is a contradiction.
We are left with the case in which $|G^{\shortmid}|=p^2$. By Corollary~\ref{cor:Hall}, $Z(G)\subseteq G^{\shortmid}$. There are two possibilities. Either, $Z(G)$ is a proper
subgroup of $G^{\shortmid}$ of cardinality $p$ or $Z(G)=G^{\shortmid}$. We will show that the second possibility never happens.
Assume $Z(G)=G^{\shortmid}$. Then, by Proposition~\ref{prop:eq}, $G^{\shortmid}=Z(G)\cong\hat{G}$. Obviously, this cannot happen if $Z(G)$ cyclic.
Indeed, in this case $G/Z(G)$ is cyclic and hence the group $G$ is abelian which contradicts our assumption that the semi-center is commutative.
Assume $Z(G)\cong C_p\times C_p$.
We get a central extension 
\begin{equation}\label{eq:centex2}
1\rightarrow G^{\shortmid}\cong  C_p\times C_p \rightarrow G \xrightarrow{\phi} C_p\times C_p \rightarrow 1.
\end{equation}
Let $x,y\in G$ such that $\phi (x),\phi(y)$ generates Im$(\phi )$. Then the commutator $G^{\shortmid}$ is the cyclic subgroup of order $p$
generated by $[x,y]$. This contradicts~\eqref{eq:centex2}.
\end{proof}

We are now ready to prove the uniqueness of $G_{(\text{xv})}$ in the following sense.
\begin{proposition}\label{th:p^4}
Let $G$ be group of cardinality $p^4$ where $p$ is prime. If there exists a nondegenerate cocycle $f \in Z^2(G,\mathbb{C})$
such that Sz$(\mathbb{C}^fG)$ is commutative then $p>2$ and $G\cong G_{(\text{xv})}$.
\end{proposition}
\begin{proof}
Clearly, $G$ cannot be abelian.
Let $G$ be a group of central type which is not isomorphic to $G_{(\text{xv})}$ from Table~\ref{tab:p4}.
Then by Theorem 1, if $p$ is odd then $G$ isomorphic to either $G_{(\text{viii})}$ or $G_{(\text{xiv})}$ from Table~\ref{tab:p4},
and if $p=2$ then by Theorem 2, $G$ is isomorphic to either $G_{(\text{ix})}$  or $G_{(\text{x})}$ from Table~\ref{tab:16}.
However, in all these groups the center is not a proper subgroup of the commutator and hence by Lemma~\ref{lemma:centproppp}
there is no commutative semi-center of simple twisted group algebras over these groups.
\end{proof}
Let $\C ^fG$ be a simple twisted group algebra. By Remark~\ref{remark:sf} and by Proposition~\ref{th:p^4}, if either 
$|G|$ is cube-free or $|G|=16$, then Sz$(\mathbb{C}^fG)$ is non-commutative. 
Hence, if a group $G$ admits a nondegenerate cocycle
$f\in Z^2(G,\C ^*)$  such that Sz$(\C ^fG)$ is commutative, then $|G|\geq 64$.

Next we prove Theorem 5.
Recall that $G$ defined as follows:
\begin{equation}
G=\langle x_1,x_2,x_3,x_4,x_5,x_6 \rangle,
\end{equation}
such that $ x_1,x_2,x_3$ are central and 
$$x_i^2=1,\quad [x_4,x_5]=x_1,\quad [x_4,x_6]=x_2,\quad [x_5,x_6]=x_3.$$
We show that there exists a nondegenerate cocycle $f \in Z^2(G,\C^*)$
such that Sz$(\mathbb{C}^fG)$ is commutative.\\
\centerline{\textbf{Proof of Theorem 5.}}
The idea of the construction of $G$ is based on Proposition~\ref{prop:eq}.
Clearly,
$$Z(G)=\langle x_1,x_2,x_3 \rangle = G^{\shortmid}.$$
Therefore, 
$$\hat {G}\cong G^{\shortmid}=Z(G).$$
The idea of the construction of $[f]\in H^2(G,\C^*)$ is based on Lemma~\ref{lemma:capG}.
We wish to construct a nondegenerate cohomology class $[f]\in H^2(G,\C^*)$ such that the restriction of $[f]$ to $G^{\shortmid}$ is trivial.
Let $u_i:=u_{x_i}$. As before, we assume that $u_i^2=1$.
We use a construction of crossed product in the following way.
Start from the commutative group algebra $R_1=\C Z(G)$ and construct a $C_2=\langle x_4 \rangle$ action on $R_1$ in the following way.
\begin{align*}
\psi _4: & R_1\rightarrow R_1\\
& u_1\mapsto t_{14}u_1\\
&          u_2\mapsto t_{24}u_2\\
 &          u_3\mapsto t_{34}u_3
\end{align*}
For any $t_{ij}=\pm 1$, $\psi _4$ is an automorphism (of order $2$) of the ring $R_1$.
Hence, the ring $R_2$ generated by $u_1,u_2,u_3,u_4$ is a crossed product of $C_2$ over $R_1$.
In particular it is an associative algebra.
Consider the following $C_2=\langle x_5 \rangle$-action on $R_2$.
\begin{align*}
\psi _5: & R_2\rightarrow R_2\\
& u_1\mapsto u_1\\
&          u_2\mapsto t_{25}u_2\\
 &          u_3\mapsto t_{35}u_3\\
 &          u_4\mapsto \delta u_1u_4
\end{align*}
We show that again, for any $t_{ij},\delta=\pm 1$, $\psi _5$ is an automorphism (of order $2$) of the ring $R_2$.
Clearly, the restriction of $\psi _5$ to $R_1$ is an automorphism. We need to check if the following equality holds for any $1\leq i\leq 3$.
\begin{equation}
\psi _5(u_4 u_i u_4^{-1})= \psi _5(u_4)\psi_5( u_i)\psi_5( u_4)^{-1}.
\end{equation}
Indeed,
$$\psi _5(u_4 u_i u_4^{-1})=\psi _5(t_{i4}u_i)=t_{i4}t_{i5}u_i,$$
and on the other hand,
\begin{align*}
& \psi _5(u_4)\psi_5( u_i)\psi_5( u_4)^{-1}=(\delta u_1u_4)( t_{i5}u_i)( \delta ^{-1}u_4^{-1}u_1^{-1})  \\
& =t_{i5}u_1(u_4u_iu_4^{-1})u_1^{-1}= t_{i4}t_{i5}u_i.
\end{align*}
Since $\psi _5$ is an automorphism of the ring $R_2$ we conclude that the ring $R_3$ generated by $u_1,u_2,u_3,u_4,u_5$ is a crossed product of $C_2$ over $R_2$.
In particular it is an associative algebra.
Finally, we construct a $C_2=\langle x_6 \rangle$-action on $R_3$ in the following way.
\begin{align*}
\psi _6: & R_3\rightarrow R_3\\
& u_1\mapsto t_{16}u_1\\
&          u_2\mapsto u_2\\
 &          u_3\mapsto u_3\\
 &          u_4\mapsto \gamma u_2u_4\\
 &          u_5\mapsto \lambda u_3u_5
\end{align*}
Now, we search conditions on $t_{ij},\delta,\gamma,\lambda=\pm 1$, such that $\psi _6$ will be an automorphism of the ring $R_3$ .
Similarly to the proof that $\psi _5$ is an automorphism of $R_2$, the restriction of $\psi _6$ to $R_2$ is also an automorphism.
We need to check if the following equality holds for any $1\leq i\leq 4$.
\begin{equation}
\psi _6(u_5 u_i u_5^{-1})= \psi _6(u_5)\psi_6( u_i)\psi_6( u_5)^{-1}.
\end{equation}
For $1\leq i\leq 3$ we have on one side
$$\psi _6(u_5 u_i u_5^{-1})=\psi _6(t_{i5}u_i)=t_{i5}t_{i6}u_i.$$
And on the other hand
$$\psi _6(u_5)\psi_6( u_i)\psi_6( u_5)^{-1}=\lambda u_3u_5 t_{i6}u_i  u_5^{-1}u_3^{-1}\lambda ^{-1}=t_{i6}u_3t_{i5}u_iu_3^{-1}=t_{i5}t_{i6}u_i.$$
For $i=4$ we get on one side
$$\psi _6(u_5 u_4 u_5^{-1})=\psi _6(\delta u_1u_4)=\delta t_{16}u_1\gamma u_2u_4.$$
On the other hand
\begin{align}
& \psi _6(u_5)\psi_6( u_4)\psi_6( u_5)^{-1}=(\lambda u_3u_5)( \gamma u_2u_4)( u_5^{-1}u_3^{-1}\lambda ^{-1})=\\
& \gamma u_3u_5u_2(u_5^{-1}u_5)u_4u_5^{-1}u_3^{-1}=\gamma u_3 t_{25}u_2\delta u_1u_4u_3^{-1}= \nonumber \\
& \gamma \delta t_{25}u_1u_2 u_3 u_4u_3^{-1}=\gamma \delta t_{25}u_1 t_{34} u_1u_2 u_4.
\end{align}
Therefore, equality holds if and only if
\begin{equation}\label{eq:eqaut}
t_{16}=t_{25}t_{34}.
\end{equation}
For $t_{16},t_{25},t_{34}\in \{1,-1\}$ which satisfy~\eqref{eq:eqaut}, $\psi _6$ is an automorphism
and then the ring $R$ generated by $u_1,u_2,u_3,u_4,u_5,u_6$ is a crossed product of $C_2$ over $R_3$.
In particular it is an associative algebra.

Let $t_{16}=t_{34}=t_{24}=t_{35}=-1$ and all the other $t_{ij}=1$ and $\lambda =\gamma=\delta=1$.
Then, the ring $R$ is a complex associative algebra generated by $u_1,u_2,u_3,u_4,u_5,u_6$ with the following relations.
\begin{equation}
[u_1,u_6]=[u_2,u_4]=[u_3,u_4]=[u_3,u_5]=-1.
\end{equation}
For all the other $1\leq i \leq 3, 4\leq j\leq 6$ we have $[u_i,u_j]=1$.
Also for all $1\leq k\leq 6$, $[u_i,u_{i^{\shortmid}}]=1$ for all $1\leq i,i^{\shortmid}\leq 3$,
$[u_4,u_5]=u_1$, $[u_4,u_6]=u_2$ and $[u_5,u_6]=u_3$.
The associativity of $R$ ensures that there exists a cocycle $f\in Z^2(G,\C ^*)$ which satisfies the above relations.
We need to show the cocycle $f$ is nondegenerate.
It easy to show that there are no $f$-regular elements in $G$ and hence $f$ is nondegenerate.
Indeed, for $g\not \in Z(G)$ there exists an element $h\in Z(G)$ such that $[u_g,u_h]\neq 1$.
On the other hand, for any $h\in Z(G)$ there exists an element $g\not \in Z(G)$ such that $[u_g,u_h]\neq 1$.
Since any element in $Z(G)$ is $(\lambda,f)$-regular and since $|Z(G)|=|\hat{G}|$ we can conclude that
all the $(\lambda,f)$-regular elements in $G$ are central and hence contained in $G^{\shortmid}=Z(G)$.
By Corollary~\ref{cor:Hall}, Sz$(\mathbb{C}^fG)$ is commutative. \qed

By Proposition~\ref{th:p^4} and Remark~\ref{remark:sf} the group $G$ presented in Theorem 5
is minimal with the property that $(\mathbb{C}^fG)$ is simple and Sz$(\mathbb{C}^fG)$ is commutative.
\section{Problem 1 for general cocycles}
As already mentioned, if Sz$(\mathbb{C}^fG)$ is simple then $\mathbb{C}^fG$ is also simple.
As for the commutative semi-center, there are many obvious examples.
Clearly, if $G$ is an abelian group then Sz$(\C G)=\C G$ is commutative.
Since the cohomology of cyclic groups is trivial, then another family of examples is any cocycle $f\in Z^2(G,\C ^*)$ on a group $G$ such that $\hat{G}$ is cyclic.

Next, we present an example of a group $G$ of order $8$ such that $\hat{G}$ is not cyclic,
a cocycle $f\in Z^2(G,\C ^*)$ such that $\C ^fG$ is non-commutative and still
Sz$(\mathbb{C}^fG)$ is commutative.
\begin{example}
Let
 $$G=D_4=\{\sigma , \tau | \sigma ^4=1, \tau ^2=1, \tau \sigma
\tau =\sigma ^{-1}\}.$$
Recall that the
conjugation classes of $D_4$ are
\begin{equation*}
\{1\},\quad \{\sigma ^2\},\quad \{\sigma, \sigma ^3\},\quad \{\sigma
\tau, \sigma ^3 \tau\},\quad \{\tau, \sigma ^2 \tau\},
\end{equation*}
and 
$$G^{\shortmid}=\{1,\sigma ^2\}.$$
Let $\mathbb{C}^fG$ be the twisted group
algebra spanned by the elements $\{u_g\}_{g\in G}$ where
\begin{equation}\label{eq:totototo}
u_{\tau}u_{\sigma}u_{\tau}^{-1}=iu_{\sigma ^3}.
\end{equation}
By~\eqref{eq:totototo}, $[u_{\sigma ^2},u_{\tau}]=-1$.
Therefore, $\sigma ^2$ is not $f$-regular and in the notations of \cite[Corollary 2.3]{ginosar2012semi}, 
$$\Gamma (1,f)\cap G^{\shortmid}=\{1\}.$$
Hence, by \cite[Corollary 2.3]{ginosar2012semi}
$$\text{dim}_{\C}(\text{Sz}(\mathbb{C}^fG))=4.$$
Consequently, Sz$(\mathbb{C}^fG)$ it either commutative or simple. Since the
semi-center of non-simple twisted group algebra is never simple we conclude that the Sz$(\mathbb{C}^fG)$ is commutative.
\end{example}

\begin{landscape}
    \begin{table}[htb!]
        \centering
       \caption{Groups of order $p^4$ where $p$ is an odd prime} 
        \setlength{\tabcolsep}{2pt}
        \begin{tabular}{l c l}
 \hline\hline 
\\ 
\hline 
$G_{(\text{i})}$ & $C_{p^4}$ &    \\
\hline 
$G_{(\text{ii})}$ & $C_{p^3}\times C_p$ &    \\
\hline 
$G_{(\text{iii})}$ & $C_{p^2}\times C_{p^2}$ &    \\
\hline 
$G_{(\text{iv})}$ & $C_{p^2}\times C_p\times C_p$ &\\
\hline 
$G_{(\text{v})}$ & $C_p^4$ &  \\
\hline
$G_{(\text{vi})}$ & $\langle a,b\rangle$ & $a^{p^3}=b^p=1,\quad aba^{-1}=ba^{p^2}$ \\
\hline 
$G_{(\text{vii})}$ & $\langle a,b,c \rangle$ & $a^{p^2}=b^p=c^p=1,\quad  [a,b]=[a,c]=1,[b,c]=a^p$ \\
\hline 
$G_{(\text{viii})}$ & $\langle a,b \rangle$ & $a^{p^2}=b^{p^2}=1,\quad  [a,b]=a^p$ \\
\hline
$G_{(\text{ix})}$ & $\langle a,b,c \rangle$ & $a^{p^2}=b^p=c^p=1,\quad  [a,b]=[b,c]=1, [a,c]=a^p$ \\
\hline 
$G_{(\text{x})}$ & $\langle a,b,c \rangle$ & $a^{p^2}=b^p=c^p=1,\quad  [a,b]=[b,c]=1, [a,c]=b$ \\
\hline 
$G_{(\text{xi})}$ for $p>3$ & $\langle a,b,c \rangle$ & $a^{p^2}=b^p=c^p=1,\quad  [b,c]=1, [a,b]=a^p,ac=cab$ \\
\hline
$G_{(\text{xii})}$ for $p>3$ & $\langle a,b,c \rangle$ & $a^{p^2}=b^p=c^p=1,\quad  [b,c]=a^p, [a,b]=a^p,ac=cab$ \\
\hline
$G_{(\text{xiii})}$ for $p>3$ & $\langle a,b,c \rangle$ & $a^{p^2}=b^p=c^p=1,\quad  [b,c]=a^{\alpha p}, [a,b]=a^p,ac=cab$, $\alpha=\text{any non-residue } (\text{mod } p)$ \\
\hline
$G_{(\text{xi})}$ for $p=3$ & $\langle a,b,c \rangle$ & $a^9=b^3=c^3=1,\quad  [b,c]=1, [a,b]=a^3,[a,c]=ba^3$ \\
\hline
$G_{(\text{xii})}$ for $p=3$ & $\langle a,b,c \rangle$ & $a^9=b^3=1,c^3=a^3,\quad  [b,c]=1, [a,b]=a^p,[c,a]=ba^3$ \\
\hline
$G_{(\text{xiii})}$ for $p=3$ & $\langle a,b,c \rangle$ & $a^9=b^3=1,c^3=a^6,\quad  [b,c]=1, [a,b]=a^p,[c,a]=ba^3$ \\
\hline
$G_{(\text{xiv})}$  & $\langle a,b,c,d \rangle$ & $a^p=b^p=c^p=d^p=1,\quad  [a,b]=[a,c]=[a,d]=[b,c]=[b,d]=1, [c,d]=a$ \\
\hline
$G_{(\text{xv})}$ for $p>3$ & $\langle a,b,c,d \rangle$ & $a^p=b^p=c^p=d^p=1,\quad  [a,b]=[a,c]=[a,d]=[b,c]=1,[d,b]=a, [d,c]=b$ \\
\hline
$G_{(\text{xv})}$ for $p=3$ & $\langle a,b,c \rangle$ & $a^9=b^3=c^3=1,\quad  [a,b]=1,[a,c]=ba^3,[b,c]=a^3$ \\
\hline
        \end{tabular}
          \label{tab:p4}
    \end{table}
\end{landscape}

\begin{landscape}
    \begin{table}[htb!]
        \centering
       \caption{Groups of order $16$} 
        \setlength{\tabcolsep}{2pt}
        \begin{tabular}{l c l}
 \hline\hline 
\\ 
\hline 
$G_{(\text{i})}$ & $C_{16}$ &    \\
\hline 
$G_{(\text{ii})}$ & $C_{8}\times C_2$ &    \\
\hline 
$G_{(\text{iii})}$ & $C_{4}\times C_{4}$ &    \\
\hline 
$G_{(\text{iv})}$ & $C_{4}\times C_2\times C_2$ &\\
\hline 
$G_{(\text{v})}$ & $C_2^4$ &  \\
\hline
$G_{(\text{vi})}$ & $\langle a,b\rangle$ & $a^{8}=b^2=1,\quad aba^{-1}=ba^4$ \\
\hline 
$G_{(\text{vii})}$ & $\langle a,b,c \rangle$ & $a^{4}=b^2=c^2=1,\quad  [a,b]=[a,c]=1,[b,c]=a^2$ \\
\hline 
$G_{(\text{viii})}$ & $\langle a,b \rangle$ & $a^4=b^4=1,\quad  [a,b]=a^2$ \\
\hline
$G_{(\text{ix})}$ & $\langle a,b,c \rangle$ & $a^4=b^2=c^2=1,\quad  [a,b]=[b,c]=1, [a,c]=a^2$ \\
\hline 
$G_{(\text{x})}$ & $\langle a,b,c \rangle$ & $a^4=b^2=c^2=1,\quad  [a,b]=[b,c]=1, [a,c]=b$ \\
\hline 
$G_{(\text{xi})}$  & $\langle a,b,c \rangle$ & $a^4=c^2=1,a^2=b^2\quad  [a,c]=[b,c]=1, [a,b]=a^2$ \\
\hline
$G_{(\text{xii})}$ & $\langle a,b \rangle$ & $a^8=b^2=1,\quad  [a,b]=a^2$ \\
\hline
$G_{(\text{xiii})}$  & $\langle a,b \rangle$ & $a^8=b^2=1,\quad  [a,b]=a^6$ \\
\hline
$G_{(\text{xiv})}$  & $\langle a,b \rangle$ & $a^8=1, b^2=a^4 \quad  [a,b]=a^6$ \\
\hline
        \end{tabular}
          \label{tab:16}
    \end{table}
\end{landscape}

\end{document}